\documentclass[11pt,a4paper,reqno]{amsart}

\usepackage[DIV=9,oneside,BCOR=0mm]{typearea}
\usepackage{microtype}
\usepackage[T1]{fontenc}

\usepackage{amsmath,amsthm,amssymb,amsfonts,xcolor,todonotes} 
\usepackage[colorlinks,citecolor=blue,urlcolor=black,linkcolor=blue,linktocpage]{hyperref}

\newtheorem{innercustomgeneric}{\customgenericname}
\providecommand{\customgenericname}{}
\newcommand{\newcustomtheorem}[2]{%
	\newenvironment{#1}[1]
	{%
		\renewcommand\customgenericname{#2}%
		\renewcommand\theinnercustomgeneric{##1}%
		\innercustomgeneric
	}
	{\endinnercustomgeneric}
}

\newcustomtheorem{customthm}{Theorem}

\newtheorem{theorem}{Theorem}[section]
\newtheorem{lemma}[theorem]{Lemma}
\newtheorem{proposition}[theorem]{Proposition}
\newtheorem{corollary}[theorem]{Corollary}

\theoremstyle{definition}
\newtheorem{definition}[theorem]{Definition}
\newtheorem{question}[theorem]{Question}
\newtheorem{remark}[theorem]{Remark}
\newtheorem{example}[theorem]{Example}
\newtheorem{examples}[theorem]{Examples}

\DeclareMathOperator{\cha}{char}
\DeclareMathOperator{\GL}{GL}

\DeclareMathOperator{\RUCB}{RUCB}

\DeclareMathOperator{\supp}{supp}
\DeclareMathOperator{\Bin}{Bin}
\DeclareMathOperator{\U}{U}
\DeclareMathOperator{\Aut}{Aut}
\DeclareMathOperator{\graph}{graph}
\DeclareMathOperator{\Sym}{Sym}

\DeclareMathOperator{\IET}{IET}
\DeclareMathOperator{\A}{C}

\renewcommand{\epsilon}{\varepsilon}

\renewcommand{\phi}{\varphi}

\makeatletter
\def\moverlay{\mathpalette\mov@rlay}
\def\mov@rlay#1#2{\leavevmode\vtop{%
		\baselineskip\z@skip \lineskiplimit-\maxdimen
		\ialign{\hfil$\m@th#1##$\hfil\cr#2\crcr}}}
\newcommand{\charfusion}[3][\mathord]{
	#1{\ifx#1\mathop\vphantom{#2}\fi
		\mathpalette\mov@rlay{#2\cr#3}
	}
	\ifx#1\mathop\expandafter\displaylimits\fi}
\makeatother

\newcommand{\cupdot}{\charfusion[\mathbin]{\cup}{\cdot}}

\begin{document}

\title[Amenability, Kesten's property, and measurable lamplighters]{Amenable equivalence relations, Kesten's property, and measurable lamplighters}

\author{Maksym Chaudkhari}
\address{M.C., Department of Mathematics and Statistics, University of South Florida, 4202 E Fowler Ave, CMC 342, Tampa, FL 33620, USA}
\email{mc637@usf.edu, maksymchaudkhari@utexas.edu}
\author{Kate Juschenko}
\address{K.J., Department of Mathematics, University of Texas at Austin, 2515 Speedway C1200, Austin, TX 78712, USA}
\email{katie.juschenko@gmail.com}
\author{Friedrich Martin Schneider}
\address{F.M.~Schneider, Institute of Discrete Mathematics and Algebra, TU Bergakademie Freiberg, 09596 Freiberg, Germany}
\email{martin.schneider@math.tu-freiberg.de}
    
\begin{abstract} We prove a characterization of the amenability of countable Borel equivalence relations in terms of the uniform Liouville property for group actions on their classes. Furthermore, inspired by a well-known amenability criterion for locally compact groups due to Kesten, we study return probabilities for random walks, and in particular a limiting condition that we call \emph{Kesten's property}, on general topological groups. We show that every amenable topological group with small invariant neighborhoods indeed has Kesten's property. For measurable lamplighter groups associated with countable Borel equivalence relations, we establish a connection between Kesten's property and anti-concentration inequalities for the inverted orbits of random walks on the equivalence classes. This allows us to construct an amenable contractible Polish group without Kesten's property. \end{abstract}

\maketitle

\tableofcontents

\newpage
	
\section{Introduction}

This paper develops connections between the amenability of topological groups and certain probabilistic properties of their countable  
subgroups. At the center of our study are full groups of countable Borel equivalence relations, along with their discrete counterparts. Among the properties under consideration are the \emph{Liouville property} and the growth of \emph{inverted orbits}. We briefly recall the relevant definitions and results for the Liouville property below, while inverted orbits and the related notion of extensive amenability are discussed in detail in Section~\ref{section:extensive.amenability}.

Let $G$ be a countable group acting on a set $X$. A probability measure on $G$ is called \emph{non-degenerate} if its support generates $G$ as a semigroup. Let $\mu$ be some symmetric non-degenerate probability measure on $G$. A function  $f \colon X \rightarrow \mathbb{R}$ is called \emph{$\mu$-harmonic} if the  equality \begin{displaymath}
    f(x) \, = \, \sum\nolimits_{g \in G} f(gx)\mu(g)
\end{displaymath} holds for every $x \in X$. The action is called \emph{$\mu$-Liouville} if every bounded $\mu$-harmonic function on $X$ is constant. We will say that the action $G \curvearrowright X$ is \emph{Liouville} if it is $\mu$-Liouville for some symmetric non-degenerate probability measure $\mu$ on $G$.
	
A classical theorem of Kaimanovich and Vershik from~\cite{KV} states that a discrete group $G$ is amenable if and only if the left multiplication action of $G$ on itself is Liouville. A version of this theorem valid for locally compact second countable groups was obtained by Rosenblatt in~\cite{Rt}. Moreover,  recently this result was extended to general second-countable topological groups in~\cite{ST}.
 
It is easy to see that if the left multiplication action of $G$ on itself is $\mu$-Liouville, then any transitive action of $G$ is also $\mu$-Liouville. Thus, the Kaimanovich--Vershik theorem implies that non-amenability of a group may be proved by constructing a non-Liouville transitive action or a family of actions that do not admit a common non-degenerate symmetric probability measure on $G$ making all of them Liouville. This idea was used by Kaimanovich, \cite{Kaimanovich}, as a suggested approach to show non-amenability of Thompson's group $F$. In particular, Kaimanovich showed that, for every finitely supported non-degenerate measure $\mu$ on Thompson's
group $F$, the action of $F$ on dyadic rationals is not $\mu$-Liouville. In \cite{JZ}, Zheng and the second author showed that this action is Liouville. Moreover, \cite{J23} raised the question whether for any natural number $n$ the action of Thompson's group $F$ on sets of dyadic rationals of cardinality $n$ is Liouville. This question was solved using topological group theory in~\cite{ST}.

In this paper, we prove an analogue of the Kaimanovich--Vershik theorem for orbit equivalence relations that also partially answers the following natural question.

\begin{question}\label{Liouville} Let $G$ be a countable group acting on a standard Borel space $X$, equipped with a non-atomic Borel probability measure $\mu$,  in a Borel way such that $\mu$ is quasi-invariant, and the induced orbit equivalence relation is $\mu$-amenable. Under natural additional assumptions on $G$ and $\mu$, does there exist a non-degenerate symmetric probability measure $\nu$ on $G$ such that action of $G$ on almost every orbit in $X$ is $\nu$-Liouville? \end{question}

\begin{customthm}{A}[Theorem~\ref{UnifLiouville}]\label{theorem:a} Let $X$ be a standard Borel space equipped with a non-atomic Borel probability measure $\mu$, let $R$ be a countable Borel equivalence relation on $X$ such that $\mu$ is $R$-quasi-invariant, and let $G$ be a countable dense subgroup of the full group $[R]$ equipped with the uniform topology. The following are equivalent. \begin{enumerate}
	\item[$(1)$] $R$ is $\mu$-amenable.
	\item[$(2)$] There exists a symmetric non-degenerate measure $\nu$ on $G$ such that the action of $G$ on almost every orbit in $X$ is $\nu$-Liouville.
\end{enumerate} \end{customthm}

The assumption that $G$ is dense in the full group of the corresponding orbit equivalence relation is satisfied, in particular, for the natural actions of the topological full groups on the Cantor space. This class of groups contains prominent examples of groups with open amenability problem, and in particular, a large family of groups of interval exchange transformations acting on the unit circle, see~\cite{JMBMdlS}. As we mentioned above, the classical Kaimanovich--Vershik theorem and its corollaries imply that the uniform Liouville property for a family of actions as stated in Question~\ref{Liouville} and in Theorem~\ref{theorem:a}(2) is a natural property to explore for a group with open amenability problem. Indeed, if a countable group $G$ admits an action on a standard probability space that fails condition in Theorem~\ref{theorem:a}(2), then it cannot be amenable by the Kaimanovich--Vershik theorem.

Moreover, a positive answer to Question~\ref{Liouville} should be expected only under additional assumptions on the group $G$ or on the measure $\mu$ (for example, one could consider only invariant measures). Indeed, a well-known result of Zimmer~\cite{Zimmer} shows that the action of the free group on its Poisson boundary corresponding to the simple random walk induces an amenable equivalence relation. However, since this action is essentially free, it cannot be Liouville on almost every orbit. 

Finally, as a corollary of Theorem~\ref{theorem:a}, we construct a family of group actions of a non-amenable group $G$ such that each action in this family is Liouville (and even $n$-Liouville for all $n \geq 1$, in the sense of~\cite{J23}), but there is no symmetric non-degenerate measure $\nu$ on $G$ that can make all of these actions $\nu$-Liouville, see Corollary~\ref{cor3}.

In a different context, amenability of the equivalence relation is also connected to the Liouville property of the equivalence relation by a result of Buehler and Kaimanovich~\cite{BK}. However, they construct random environments which have the Liouville property but may not be space-homogeneous, hence the setting and the result is different from the case considered in our article.

The second main theme of the article is the study of the generalizations of Kesten's theorem to amenable topological groups and the implications of these results for the behavior of countable subgroups of amenable topological groups. In particular, we will connect Kesten's theorem for a class of topological groups with extensive amenability of the actions of countable groups on the orbits of an amenable equivalence relation.

Our main positive result in the direction of generalizing Kesten's theorem is the following theorem for amenable topological groups with \emph{small invariant neighborhoods}~(\emph{SIN}), see Section~\ref{Preliminaries} for definitions.

\begin{customthm}{B}[Theorem~\ref{SINKesten}] Let $G$ be an amenable topological group with SIN and let $\nu$ be a symmetric regular Borel probability measure on $G$ with countable support. Then, for any open identity neighborhood $U$ in $G$, \begin{displaymath}
    \limsup\nolimits_{n \rightarrow \infty}\nu^{n}(U^n)^{1/n} \, = \, 1.
\end{displaymath} \end{customthm}

This theorem may be viewed as a combinatorial extension of Kesten's theorem to a class of non-locally compact topological groups. It may also be interpreted as a statement about the rate of escape for random walks in topological groups. Moreover, we provide examples of amenable topological groups with small invariant neighborhoods illustrating that stronger versions of this statement fail. 

Furthermore, the main results of Section~\ref{section:lamplighters}---Theorem~\ref{Top.Lamplighter}, Theorem~\ref{inv.orb.T}, and Corollary~\ref{counterexample_K}---allow us to construct a counterexample to the generalized version of Kesten's theorem in the most natural non-SIN case, where the SIN property is replaced by the local generation assumption. These results also relate the generalization of Kesten's theorem for topological groups with behavior of inverted orbits of random walks in the discrete setting. The counterexample, described in Corollary \ref{counterexample_K}, belongs to a family of amenable topological groups that we will call \emph{measurable lamplighter groups}. We note that we do not know if the extension of Kesten's theorem holds for several other representatives of this family. Moreover, we expect a positive answer for them, and it could help us to understand the behavior of inverted orbits of points for the random walks on the orbits of an ergodic amenable probability measure preserving equivalence relation. We will briefly describe the main examples in the next paragraph and we refer the reader to Section~\ref{section:lamplighters} for more details.

Let $R$ be a countable Borel equivalence relation on a standard Borel space $X$ with a non-atomic probability measure $\mu$ such that $\mu$ is $R$-quasi-invariant. Then $R$ carries the Borel structure inherited from $X \times X$ as well as a measure $\mu_{R}$ induced by $\mu$. Let us denote by $\A_{\mu}(R)$ the group of all $\mu_{R}$-equivalence classes of Borel subsets of $R$ of finite $\mu_{R}$ measure, with the group operation defined by the symmetric difference of sets. This group comes equipped with a natural distance, which generates what we call the \emph{$L^{1}$-topology} on $\A_{\mu}(R)$, and the full group $[R]$ acts on $\A_{\mu}(R)$ by isometries. The corresponding semidirect product $\A_{\mu}(R)\rtimes[R]$ may be viewed as an analogue of the lamplighter group obtained from an action of a discrete group on a countable set. We prove the following.

\begin{customthm}{C}[Theorem~\ref{Top.Lamplighter}] Let $R$ be a countable Borel equivalence relation on a standard Borel space $X$ with a non-atomic Borel probability measure $\mu$. Let $[R]$ be endowed with the uniform topology and $\A_{\mu}(R)$ with the $L^{1}$-topology. Then $\A_{\mu}(R) \rtimes [R]$ with the product topology is a Polish topological group. Moreover, the following hold. \begin{enumerate}
	\item[$(1)$] If $\mu$ is $R$-invariant, then $\A_{\mu}(R) \rtimes [R]$ admits a complete left-invariant metric compatible with its topology.
    \item[$(2)$] If $R$ is $\mu$-amenable, then the topological group $\A_{\mu}(R) \rtimes[R]$ is amenable.
    \item[$(3)$] If $\mu$ is $R$-ergodic, then $\A_{\mu}(R) \rtimes[R]$ is contractible and does not have SIN.
\end{enumerate} \end{customthm} 

The importance of semidirect products for amenability was realized in~\cite{JM}, \cite{JNdlS}. The authors of these articles showed that under certain conditions amenability of an action of a lamplighter group implies amenability of the background group. In this paper we lift these results to topological setting. In particular, if Kesten's theorem  holds for $\A_{\mu}(R) \rtimes[R]$, our Theorem~\ref{inv.orb.T} allows us to obtain an anti-concentration inequality for the average size of an inverted orbit of a point for the random walks on the orbits of $R$. The study of inverted orbits is closely related to the growth of groups and to the notion of extensive amenability of group actions.

Extensive amenability was first formally defined in~\cite{JMBMdlS} but was used implicitly in~\cite{JM} and in~\cite{JNdlS}, where it enabled several breakthroughs in the study of amenability of discrete groups. Extensive amenability of an action admits a characterization in terms of the anti-concentration inequality for the sizes of inverted orbits (see~\cite[Proposition~4.1]{JMBMdlS}). A major open problem of amenability of the group of interval exchange transformations was also reduced to verifying extensive amenability of its action on the unit circle in~\cite{JMBMdlS}. Our Theorem~\ref{inv.orb.T} may be viewed as a step towards establishing extensive amenability of the action of the IET group on the unit circle.

Finally, the connection between topological versions of Kesten's theorem and the behavior of inverted orbits is precisely the tool we need to provide an example of an amenable contractible Polish topological group failing the generalization of Kesten's theorem. The example we construct in Corollary~\ref{counterexample_K} is also a measurable lamplighter over an amenable but not measure preserving equivalence relation.

\subsection{Organization} The structure of this article is as follows. Section~\ref{Preliminaries} contains a brief overview of the background material on topological groups and countable Borel equivalence relations. In Section~\ref{ULiouville} we prove our main results concerning the Liouville property, Theorem~\ref{UnifLiouville} and Corollary~\ref{cor3}. Section~\ref{TopKesten} contains the proof of Theorem~\ref{SINKesten} and the discussion on possible extensions of Kesten's theorem to topological groups. The subsequent Section~\ref{section:extensive.amenability} provides some relevant background on extensive amenability and inverted orbits. In Section~\ref{section:lamplighters} we discuss the measurable analogues of lamplighter groups and possible applications of Kesten's theorem to them. Moreover, we present several questions related to extensive amenability of group actions and inverted orbits in Section~\ref{section:extensive.amenability} and Section~\ref{section:lamplighters}.	
	
\section{Preliminaries}\label{Preliminaries}

\subsection{Amenability}

In this preliminary subsection, we briefly recall the definition of and some basic facts related to the notions of amenability for topological groups.

Let $G$ be a topological group.\footnote{In this paper we always work under the assumption that a topological group is Hausdorff.} Let $\mathcal{U}(G)$ denote the neighborhood filter of the neutral element in $G$. Then $G$ is said to have \emph{small invariant neighborhoods} (\emph{SIN}) if $\mathcal{U}(G)$ admits a filter base consisting of conjugation-invariant sets. The class of topological groups with SIN contains all groups with a topology generated by a bi-invariant metric. In particular, any discrete group has SIN. However, for locally compact groups this property is more restrictive; for instance, a connected locally compact group with SIN is a compact extension of its center, see~\cite[Theorem 4.3]{GM_67}. The reader may also find a list of topological groups with SIN in \cite[Section 3]{P_17}.

A topological topological group $G$ is called \emph{amenable} if the space \begin{displaymath}
    \RUCB(G) := \{ f \in \ell^{\infty}(G) \mid \forall \epsilon > 0 \, \exists U \in \mathcal{U}(G) \, \forall x \in G \, \forall y \in Ux \colon \, \vert f(x) - f(y) \vert  \leq \epsilon \}
\end{displaymath} of right-uniformly continuous bounded real-valued functions on $G$ admits a left-invariant mean, i.e., a positive unital linear functional $\mu \colon \RUCB(G) \to \mathbb{R}$ such that $\mu(f \circ {\lambda_{g}}) = \mu(f)$ for all $f \in \RUCB(G)$ and $g \in G$, where $\lambda_{g} \colon G \to G, \, x \mapsto gx$ for $g \in G$. By work of Rickert~\cite[Theorem~4.2]{R67}, a topological group $G$ is amenable if and only if every continuous action of $G$ by affine homeomorphisms on a non-void compact convex subset of a locally convex topological vector space admits a fixed point. We refer to~\cite{R67,P_06,GdlH,ST_F} for further characterizations and closure properties of the class of amenable topological groups.
	
A topological group $G$ is said to be \emph{extremely amenable} if every continuous action of $G$ on a non-empty compact Hausdorff space has a fixed point, or equivalently, if the algebra $\RUCB(G)$ admits a left-invariant multiplicative mean. The reader is referred to~\cite{P_06,GdlH} for further details on this property and a discussion of examples.
	
\subsection{Equivalence relations and their full groups}\label{CBER} Below we give a very brief description of the key properties of countable Borel equivalence relations, and we refer the reader to~\cite{KM} for the comprehensive treatment of the subject.
    
Let $(X,\mu)$ be a \emph{standard measure space}, that is, $X$ is a standard Borel space and $\mu$ is a non-atomic probability measure on $X$. A \emph{countable Borel equivalence relation} on $X$ is any equivalence relation $R$ on $X$ such that $R$ is a Borel subset of $X \times X$ and the equivalence classes of $R$ are countable.
	
Two countable Borel equivalence relations $R$ and $E$ on standard measure spaces $(X,\mu)$ and $(Y,\nu)$ are called  \emph{orbit equivalent} if there is a Borel measure isomorphism $\phi\colon X \rightarrow Y $ which maps equivalence classes of $R$ to equivalence classes of $E$ on an invariant set of full measure. For a countable Borel equivalence relation $R$ on $X$, we will call the Borel full group, denoted by $[R]_B$, the group of all Borel automorphisms $\phi$ of $X$ such that the graph of $\phi$ is contained in $R$. 

Let $R$ be a countable Borel equivalence relation $R$ on a standard measure space $(X, \mu)$. We will say that $\mu$ is \emph{$R$-invariant} (resp., \emph{$R$-quasi-invariant}) if $\mu$ (resp., the measure class of $\mu$) is preserved under the action of $[R]_B$ on $(X,\mu)$. An $R$-quasi-invariant probability measure $\mu$ is called \emph{$R$-ergodic} if every $R$-invariant Borel set is either null or co-null. Finally, we denote by $\Aut(X, \mu)$ the group of all Borel automorphisms of $X$ that preserve the measure class of $\mu$ in which two automorphisms that differ on a set of $\mu$ measure $0$ are identified. 

\begin{definition}\label{definition:uniform distance} Let $R$ be a countable Borel equivalence relation on a standard measure space $(X,\mu)$ such that $\mu$ is $R$-quasi-invariant. The \emph{full group} $[R]$ is defined as the group of all Borel automorphisms $\phi \in \Aut(X,\mu)$ such that $\graph(\phi) \subseteq R$ on a subset of full measure. Furthermore, the \emph{uniform distance} on $[R]$ defined as \begin{displaymath}
    d(g,h) \, := \, \mu(\{x \in X \mid g(x) \neq h(x) \}) \qquad (g,h \in [R]).
\end{displaymath} \end{definition}	
	  
If $R$ is a countable Borel equivalence relation $R$ on a standard measure space $(X,\mu)$ and $\mu$ is $R$-quasi-invariant, then the corresponding uniform distance constitutes a left-invariant metric on $[R]$ and it generates a group topology on $[R]$, which we refer to as the \emph{uniform topology}. In case that $\mu$ is even $R$-invariant, the uniform distance is both left- and right-invariant, whence the uniform topology has SIN.

\begin{remark} In the case of $R$-quasi-invariant measure $\mu$, some authors include the summand $\mu(\{x \in X \mid g^{-1}(x) \neq h^{-1}(x) \})$ in the definition of the uniform distance, see for example~\cite[pp.~920--921]{D95}. In this case, the uniform distance on $[R]$ is complete, but if $\mu$ is not $R$-invariant this metric fails to be left-invariant in  general case. Both approaches to defining the uniform distance are common in the literature, see for example, \cite[Section 4.1 and Section 5.1]{GP}, and they induce the same Polish topological group structure on $[R]$. However, in general, if $\mu$ is not $R$-invariant the metric in Definition~\ref{definition:uniform distance} is not complete. \end{remark}
	
If $R$ is measure-preserving and ergodic, its orbit equivalence class is completely determined by the isomorphism class of $[R]$  viewed as either a topological or abstract group (due to Dye's reconstruction theorem, see for example~\cite[Theorem~4.1]{K}).
	
We are going to state the definition of an amenable equivalence relation in terms of the Reiter's sequences. Amenability of equivalence relations has several equivalent definitions, and  properties of amenable groups often translate to amenable equivalence relations, in particular, there is an  analogue of the Følner condition, see~\cite[Chapter~9]{KM} for the details. In the case of $R$-invariant measure $\mu$, this condition implies amenability of bounded Borel graphs whose connected components are contained in the equivalence classes of the amenable equivalence relation.

\begin{definition} A countable Borel equivalence relation $R$ on a standard measure space $(X,\mu)$ is called \emph{$\mu$-amenable} if it admits a sequence of functions $l_n \colon R \rightarrow \mathbb{R}_{\geq 0}$ which satisfy the following conditions: \begin{enumerate}
	\item[$(1)$] $l_n$ is a Borel function.
	\item[$(2)$] $\Vert l_n(x,\cdot) \Vert_{1}=1$ for $\mu$-almost every $x \in X$.
	\item[$(3)$] $\Vert l_n(x,\cdot) -l_n(y,\cdot) \Vert_{1} \rightarrow 0$ for any $(x,y) \in R$ on an invariant Borel set of full $\mu$-measure.
\end{enumerate} \end{definition}
	
A Borel action of an amenable countable group $G$ on a standard measure space $(X,\mu)$ with quasi-invariant measure $\mu$ always induces a $\mu$-amenable equivalence relation. Moreover, this construction is exhaustive: any $\mu$-amenable equivalence relation $R$ on the standard measure space $(X,\mu)$ may be realized (on a subset of full measure) as the orbit equivalence relation of a Borel action of an amenable countable group $G$ on $(X,\mu)$. Recall that a countable Borel equivalence relation is said to be \emph{$\mu$-hyperfinite} if it can be represented as a union of an increasing sequence of finite Borel equivalence subrelations on a set of full measure. Then Dye's theorem, Connes--Feldman--Weiss theorem and Ornstein--Weiss theorem (see~\cite[Chapters 6, 7, 10]{KM}) establish the following characterization of $\mu$-amenable equivalence relations.

\begin{theorem}\label{thm: summary_d_cfw_ow} Let $R$ be a countable Borel equivalence relation on a standard measure space $(X,\mu)$ with $R$-quasi-invariant measure $\mu$. Then the following are equivalent. \begin{enumerate}
    \item[$(1)$] $R$ is $\mu$-amenable.
	\item[$(2)$] $R$ is $\mu$-hyperfinite.
	\item[$(3)$] On an invariant Borel set of full measure $R$ is induced by a Borel action of $\mathbb{Z}$.
\end{enumerate} Moreover, all ergodic measure-preserving $\mu$-hyperfinite countable Borel equivalence relations belong to the same orbit equivalence class. \end{theorem}
    
Furthermore, amenability of equivalence relations with respect to a measure is preserved under passing to subequivalence relations, taking products and passing to product measures, and taking unions of increasing sequences of equivalence relations, see, for example ~\cite[Chapters~6,~9,~10]{KM}.
	
Amenability of an equivalence relation is reflected in properties of its full group. When measure $\mu$ is invariant, non-atomic and ergodic, amenable $R$ produce the unique isomorphism class of full groups $[R]$ (in either abstract or topological group setting) and the topology on $[R]$ is Polish.  Furthermore, $\mu$-amenability of $R$ is equivalent to extreme amenability of $[R]$, equipped with the uniform distance, by a result of Pestov and Giordano (see~\cite[Theorem~5.7]{GP}, which is valid for quasi-invariant measures). 
	
For a $\mu$-amenable $R$ and  an ergodic $\mu$, the class of groups which could be densely embedded in  $[R]$ includes the  topological full groups of the actions of countable discrete amenable groups on the Cantor space and their commutator subgroups, and in particular, groups of interval exchange transformations with breakpoints in a given countable subgroup of the unit circle $\Lambda$ (denoted by $\IET(\Lambda)$).
	
\section{Uniform Liouville property} \label{ULiouville}
 
In this section we prove the following result. Note that we only assume that the measure is quasi-invariant with respect to the considered equivalence relation. 
	
\begin{theorem}\label{UnifLiouville} Let $R$ be a countable Borel equivalence relation on a standard measure space $(X,\mu)$ such that $\mu$ is $R$-quasi-invariant. Let $G$ be a dense countable subgroup of $[R]$. Then the following are equivalent. \begin{enumerate}
	\item[$(1)$] $R$ is $\mu$-amenable.
	\item[$(2)$] There exists a symmetric non-degenerate measure $\nu$ on $G$ such that the action of  $G$ on almost every orbit in $X$ is $\nu$-Liouville.
\end{enumerate} \end{theorem}
 
The proof of (1)$\Longrightarrow$(2) resembles the proof of a generalization of~\cite[Lemma~2]{JZ}, which in turn is inspired by the argument originally due to Kaimanovich and Vershik. Since this version of the lemma does not appear in the literature, and its proof is a good illustration of the method, we will state and prove it for reader's convenience. We will use the following notation in the statement of the lemma and in the subsequent sections of the article. When a group $G$ acts on a set $X$ and $\nu$ is a probability measure on $G$, we will denote by $P_{\nu}$ the transition probability of the induced random walk on $X$ which is explicitly expressed as $P_{\nu}(x,y)=\nu(\{g \in G \mid gx=y \})$. Notice that if $\nu$ is a symmetric probability measure on $G$, then for any $n \geq 1$ the convolution $P_{\nu}^n$ is equal to $P_{\nu^n}.$ Moreover, to simplify the notation, if $\theta $ is a probability measure on $X$ and $\nu$ is a probability measure on $G$, then abbreviate the convolution as $\theta \nu := \theta P_{\nu}$.
 
\begin{lemma}\label{lemma:basic} Let a countable group $G$  act transitively on the set $Y$. Suppose that there exist an increasing sequence $(K_n)_{n \in \mathbb{N}}$ of finite sets with $\bigcup_{n \in \mathbb{N}} K_{n} = Y$ and a sequence $(\epsilon_n)_{n \in \mathbb{N}}$ of real numbers decreasing to $0$ such that, for each $n \in \mathbb{N}$, there exists a symmetric probability measure $\nu_n$ on $G$ with finite support such that \begin{equation}\label{TV}
 	\sup\nolimits_{x,y \in K_n}	\|\delta_x \nu_n  - \delta_y \nu_n \|_1 \, < \, \epsilon_n.	
\end{equation} Then there exists a non-degenerate symmetric probability measure $\nu$ on $G$ such that the action of $G$ on $Y$ is $\nu$-Liouville. \end{lemma} 

\begin{proof} The proof of this lemma is constructive. The desired measure $\nu$ is defined as an infinite convex combination $$\nu \, = \, \sum\nolimits_{j \geq 0} c_j \nu_{n_j}.$$ At the first step, we choose an arbitrary symmetric non-degenerate probability measure $\nu_0$ on $G$ and an arbitrary sequence of positive reals $(c_j)_{j \geq 0}$ with $\sum_{j \geq 0} c_{j} = 1$. Then, the indices $n_j$ are defined inductively in order to guarantee that $$ \liminf\nolimits_{m\rightarrow \infty}\| \delta_x \nu^m -\delta_y \nu^m \|_1 \, = \, 0 $$ for any pair $x,y \in Y$. More precisely, one first selects a sequence $m_j$ such that $(c_0+\ldots+c_{j-1})^{m_j} \leq 1/j$ and then inductively defines an auxiliary sequence  of approximations $\theta_j $ to $\nu_0$ and the sequence  $n_j$. Once we know $n_{j-1}$, let $S_j$ be the finite set of the measures which can be represented as a convolution (in any order and possibly with repetitions) of at most  $m_j$ measures from the collection $\nu_0,\ldots,\nu_{n_{j-1}}$. Since $S_j$ is finite, one can choose $\theta_j$ as a restriction of $\nu_0$ to a sufficiently large finite subset of $G$ in a way that ensures that, if one takes any convolution $s \in S_j$ and replaces every occurrence of $\nu_0$ by $\theta_j$, then the resulting measure $s'$ satisfies the inequality $\|s-s' \|_1 \leq 1/j$. We denote the updated set of convolutions $S'_j$.  Then we choose $n_j$ so that $K_{n_j}$ contains any set of the form $gK_{n_{j-1}}$, where $g \in G$ could be any element in the support of any measure from $S'_j$ (notice that there are only finitely many such elements). With this choice of $n_j$ one can show that, for any $x,y \in K_{n_{j-1}}$ and any $s \in S_j$, \begin{align*}
 	\|\delta_x s\nu_{n_j}-\delta_ys\nu_{n_j} \|_1 \, &\leq \, \|\delta_x s'\nu_{n_j}-\delta_ys'\nu_{n_j} \|_1 +\|\delta_x (s-s')\nu_{n_j}-\delta_y(s-s')\nu_{n_j} \|_1  \\ 
 	&\leq \, \epsilon_{n_j}+ 2\| s-s'\|_1 \leq \epsilon_{n_j}+2/j .
\end{align*} In the inequalities above the convolutions of the form $\delta_x s\nu$ should be interpreted as $\delta_x P_s P_{\nu}$, and we implicitly use the inequality $$\| \delta_xP_s -\delta_x P_{s'}\|_1 \, \leq \, \|s-s' \|_1.$$ As a result, we can estimate $\|\delta_x\nu^{m_j}  -\delta_y\nu^{m_j} \|_1$ for any $x,y \in K_{n_{j-1}}$ as follows. Let $$\lambda \, := \, \sum\nolimits_{k=0}^{j-1} c_k \nu_{n_k}.$$ Then $\|\lambda^{m_j}\|_1 =(c_0+\ldots+c_{j-1})^{m_j} \leq 1/j$ and $\nu'=\nu^{m_j}-\lambda^{m_j}$ contains only the summands with at least one appearance of some measure $\nu_N$ with $N \geq n_j$. But then the choice of $n_j$ and the previous inequalities  imply that \begin{align*}
    \| \delta_x \nu^{m_j}  - \delta_y \nu^{m_j}\|_1 \, &\leq \, \| \delta_x \lambda^{m_j}  -\delta_y\lambda^{m_j}\|_1 + \| \delta_x \nu'  - \delta_y \nu' \|_1\\
    &\leq \, 2/j +(1-1/j)(\epsilon_{n_j}+ 2/j) < 4/j+\epsilon_{n_j} .
\end{align*} Therefore, since $K_n$ is increasing and exhausts $Y$, for any $x,y \in Y$ $$ \lim\nolimits_{j\rightarrow \infty}\|  \delta_x \nu^{m_j}  - \delta_y \nu^{m_j}\|_1 \, = \, 0 . $$ This finishes the proof. \end{proof}

Now we are ready to prove the theorem.

\begin{proof}[Proof of Theorem~\ref{UnifLiouville}] We first prove the implication (2)$\Longrightarrow$(1). We do not need the density assumption for this direction, only the fact that $G$ generates $R$. Assume that there exists a non-degenerate probability measure $\nu$ on $G$ such that the action on almost every $G$-orbit is $\nu$-Liouville. We may assume that $\nu$ defines a lazy random walk (a random walk driven by $\nu$ is \emph{lazy} if $\nu(1) \geq \frac{1}{2}$). Then~\cite[Lemma~1]{JZ} implies that the sequence of functions $l_n(x,y)=P_{{\nu}^n}(x,y), n \geq 1,$ where $P_{{\nu}^n}(x,y)$ denotes the probability that the random walk on $X$ induced by $v$ and started at $x$ is at the point $y$ after $n$ steps, satisfies the following conditions: \begin{enumerate}
	\item[---\,] $l_n$ is a nonnegative Borel function supported on $R$.
	\item[---\,] $\Vert l_n(x,\cdot) \Vert_{1}=1$ for $\mu$-almost every $x \in X$.
	\item[---\,] $\Vert l_n(x,\cdot) -l_n(y,\cdot) \Vert_{1} \rightarrow 0$ for any $(x,y) \in R$ on an invariant Borel set of full $\mu$-measure.
\end{enumerate} Therefore, this sequence witnesses $\mu$-amenability of $R$.
		
In order to prove the reverse implication, we will translate the construction in the proof of Lemma~\ref{lemma:basic} to the measurable setting. Recall that, by the Connes-Feldman-Weiss theorem (\cite{CFW81}, also see Theorem~\ref{thm: summary_d_cfw_ow}(3)), $\mu$-amenability of $R$ implies that there is $T \in [R]$ such that, for all $x$ in a co-null subset $A \subseteq X$, the  $T$-orbit of $x$ and its equivalence class in $R$ coincide.
		
Then, for each $x \in A$, the sequence of sets \begin{displaymath}
    \left. K_n(x)\, := \, \left\{T^{k}(x)\, \right\vert k \in \{-n,\ldots,n\}\right\} 
\end{displaymath} is increasing and $\mu$-essentially exhausts the equivalence class of $x$. Since $G$ is dense in $[R]$ and group operations are continuous, for each $n$, we find $f_n \in G$ such that \begin{displaymath}
    \left. \max \left\{ d(f_n^k,T^k) \, \right\vert -n^2-n \leq k \leq n^2+n \right\} \, \leq \, 1/2^n .
\end{displaymath} This choice guarantees that the set $$B_n \, := \, \left\{x \in X \left\vert \, \exists k \in \mathbb{Z} \colon \, -n^2-n \leq k \leq n^2+n, \, f_n^k(x)\neq T^k(x) \right\} \right.$$ satisfies $\mu(B_n) \leq n^3/2^n$ for all sufficiently large $n$. 
		
Let $\nu_n$ be the uniform measure on the set $\{f_n^k \mid k \in \{-n^2,\ldots, n^2\}\} $. Then, for $\epsilon_n \approx 1/n$, for any $x \in  X\setminus B_n$,  the set $K_n(x)$ and measure $\nu_n$ satisfy the condition~\eqref{TV} from Lemma~\ref{lemma:basic}.

Next, we fix a sequence of positive weights  $(c_j)_{j \in \mathbb{Z}_+}$ with the sum equal to 1, and construct $$\nu \, = \, \sum\nolimits_{j \geq 0} c_j \nu_{n_j}$$ as in the proof of Lemma~\ref{lemma:basic}. The sequence of indices $n_j$ is defined inductively, however in this setting we will choose $n_j$ large enough so that the set \begin{displaymath}
    D_j \, := \, \left\{x \in X \left\vert \, \exists g \in \bigcup\nolimits_{s \in S'_j} \supp(s) \colon \, gK_{n_{j-1}}(x) \nsubseteq K_{n_j}(x) \right\} \right.
\end{displaymath} has measure $\mu(D_j)<1/2^n$, where $S'_j$ is defined as in the proof of Lemma~\ref{lemma:basic}.
		
With this construction the Borel-Cantelli lemma ensures that the set of points which belong to infinitely many sets in the sequence  $(B_n)$ or $(D_j)$ has measure 0. Therefore, after removing these points together with their orbits, we get the set of the full measure where the estimates from the proof of Lemma~\ref{lemma:basic} hold for each orbit, which completes the proof of the theorem. \end{proof}
	
\begin{remark} It is easy to see that the proof of the implication (1)$\Longrightarrow$(2) works when the closure of $G$ in $[R]$ contains a discrete amenable group which generates $R$. It is also easy to see that the conditions of Lemma~\ref{lemma:basic} are satisfied whenever the closure of $G$ in $\Sym(Y)$ equipped with the topology of pointwise convergence contains a discrete amenable group acting transitively on $Y$. \end{remark} 
	
\begin{corollary}\label{cor3} Let $R$ be a non-amenable ergodic invariant countable Borel equivalence relation on a standard measure space $(X,\mu)$, and let $G$ be a dense subgroup of $[R]$. Then for $\mu$-almost every $x \in X$ there exists a symmetric non-degenerate probability measure $\nu_x$ on $G$ such that for every natural number $n$ the action of $G$ on the set of $n$-tuples of points from the orbit of $x$ is $\nu_x$-Liouville, but there is no symmetric non-degenerate probability measure $\nu$ on $G$ such that action on almost every orbit is $\nu$-Liouville. \end{corollary} 

\begin{proof} \cite[Proposition~1.19]{EG} implies that action of $G$ on almost  every orbit is highly transitive.  In particular, for almost every $x \in X$,  $G$ can be viewed as a dense subgroup of the group of all permutations on the orbit of $x$ equipped with the topology of pointwise convergence, and the latter is amenable as a topological group. Therefore, \cite[Theorem~5.8]{ST} implies that for almost every $x \in X$ one can find a symmetric  non-degenerate probability measure $\nu_x$  on $G$ such that the action of $G$ on the set of $n$-tuples of points from the orbit of $x$ is $\nu_x$-Liouville. However, our Theorem~\ref{UnifLiouville} implies that the uniform choice of the measure (finding $\nu$ which works for almost every orbit) is only possible for amenable equivalence relations. \end{proof}
			
\section{Kesten's theorem for topological groups} \label{TopKesten}

This section revolves around topological extensions of the following well-known characterization of the amenability of discrete groups due to Kesten~\cite{Kesten59}. 

\begin{theorem}[\cite{Kesten59}] Let $\Gamma$ be a finitely generated group and let $\mu$ be a finitely supported symmetric non-degenerate probability measure on $\Gamma$. Let $\rho$ be the spectral radius of the $\mu$-random walk on $\Gamma$. Then $\Gamma$ is amenable if and only if $\rho = 1$. \end{theorem}

Kesten's criterion has generalizations to graphs and networks. Below we state a quantitative version due to Mohar relating the edge-expansion of an infinite connected network and its spectral radius (see~\cite[Section 6.1 and Section 6.2, pp.~177, 184, 185]{LP} for terminology). 

\begin{theorem}[{\cite[Theorem~6.7]{LP}}]\label{Mohar} Let $(G,c,\pi)$  be a connected infinite network, let $\Phi_E$ be its edge-expansion constant, and let $\rho$ be the spectral radius of the associated network random walk. Then \begin{equation*}
	1 - \sqrt{1-\Phi_E^2} \, \leq \, 1-\rho \, \leq \, \Phi_E .
\end{equation*} \end{theorem}  

Notice that Theorem~\ref{Mohar} is valid when the network has vertices that are incident to infinitely many edges with positive weights, provided that the sum of the weights at each vertex is finite, see~\cite[Section~6.1, condition (6.1), p.~177]{LP}. Later in this section, we will apply Theorem~\ref{Mohar} to the following special type of a network. Let $\Gamma$ be a countable group acting on a set $O$ transitively, and let $\nu$ be a symmetric non-degenerate probability measure on $\Gamma$. Recall that random walk on $\Gamma$ driven by $\nu$ induces a random walk on $O$. This random walk on $O$ defines a transitive network on $O$ with the conductances given by the transition probabilities $P_{\nu}$. In this case, the isoperimetric constant $\Phi_E$ is computed as 
\[ \left. \Phi_E \, = \, \inf \left\{ \frac{|\partial_E F|_{P_\nu}}{|F|} \, \right\vert  F \subseteq O, \, F \neq \emptyset, \,|F| < \infty \right\}, \]
where for a non-empty finite set $F \subseteq O$ the weight of the edge boundary of $F$, denoted by $|\partial_E F|_{P_\nu}$, is computed as 
\[|\partial_E F|_{P_\nu} \, = \, \sum\nolimits_{x \in F, \; y \in O \setminus F} P_{\nu}(x,y).\]

In the absence of a natural analogue of the $L^2$ space associated to general topological groups, one may instead focus on the return probabilities of the considered random walks to formulate a generalization of Kesten's theorem in this context. Therefore, we consider the following two properties as extensions of Kesten's theorem to the general setting of topological groups.

\begin{definition}\label{def: Kesten property} We say that a topological group $G$ has the \begin{itemize}
    \item[---\,] \emph{Kesten property} if, for every open $U \in \mathcal{U}(G)$ and every symmetric regular Borel probability measure $\nu$ on $G$ with countable support, \begin{displaymath}
            \qquad \limsup\nolimits_{n \rightarrow \infty}\nu^{n}(U^n)^{1/n} \, = \, 1 ,
        \end{displaymath}
    \item[---\,] \emph{strong Kesten property} if, for every open $U \in \mathcal{U}(G)$ and every symmetric regular Borel probability measure $\nu$ on $G$ with countable support, \begin{displaymath}
            \qquad \limsup\nolimits_{n \rightarrow \infty}\nu^{n}(U)^{1/n} \, = \, 1 .   
        \end{displaymath}
\end{itemize} \end{definition}

In other words, a topological group $G$ has the Kesten property, respectively the strong Kesten property, if and only if, for every open $U \in \mathcal{U}(G)$ and every random walk $(X_{n})_{n \in \mathbb{N}}$ started at the identity in $G$ and governed by a symmetric regular Borel probability measure on $G$ with countable support, \begin{displaymath}
    \limsup\nolimits_{n \rightarrow \infty}\mathbb{P}(X_n \in U^n)^{1/n} \, = \, 1 ,
\end{displaymath} respectively \begin{displaymath}
    \limsup\nolimits_{n \rightarrow \infty}\mathbb{P}(X_n \in U)^{1/n} \, = \, 1 .   
\end{displaymath}

\begin{remark}\label{remark:subgroups} If a topological group $G$ has the strong Kesten property, then so does every topological subgroup of $G$. \end{remark}

The following result, essentially a consequence of~\cite{BeCh-74a}, \cite{BeCh-74b}, and \cite{Q}, clarifies the connection between amenability and the (strong) Kesten property in the locally compact setting.

\begin{proposition}\label{proposition:kesten.lc} Let $G$ be a locally compact group. The following are equivalent. \begin{enumerate}
    \item[$(1)$]\label{proposition:kesten.lc.1} $G$ is amenable.
    \item[$(2)$]\label{proposition:kesten.lc.2} $G$ has the Kesten property.
    \item[$(3)$]\label{proposition:kesten.lc.3} $G$ has the strong Kesten property.
\end{enumerate} \end{proposition}

\begin{proof} (1)$\Longrightarrow$(3): Let $G$ be an amenable locally compact group and let $\lambda$ be a left Haar measure on $G$. If $\nu$ is a symmetric measure with countable support, then the discussion on~\cite[p.~1]{BeCh-74a}  (also see~\cite[p.~48]{greenleaf}) implies that the norm of the Markov operator $M_{\nu}$ on $ L^2 (G, \lambda)$ is equal to $1$, and it is equal to $\limsup\nolimits_{n \rightarrow \infty} \nu^n (V)^{1/n} $ for any compact neighborhood $V$ of the identity by the results of \cite{BeCh-74b}. Hence, amenable locally compact topological groups satisfy the strong Kesten property.

(3)$\Longrightarrow$(2): Trivial.

(2)$\Longrightarrow$(1): We argue by contradiction. Assume that $G$ is a non-amenable group which satisfies the Kesten property. Then, one can find a compact subset $K$ of $G$ that contains some open neighborhood $V \in \mathcal{U}(G)$ such that the subgroup $H$ generated by $K$ is also non-amenable. Since $H$ is non-amenable, it contains a finite symmetric set $S$, such that the group generated by $S$ is not amenable as a discrete group. Moreover, we may also assume that $1 \in S$. Let $\nu$ be the uniform probability measure on $S$, then $\nu$ satisfies the assumptions of \cite[Theorem 1]{BeCh-74a}. An application of \cite[Corollary~7.3]{Q}\footnote{For the reader's convenience, we also include the proof of a slightly modified version of~\cite[Corollary~7.3]{Q} in the appendix of this article, see Section~\ref{lcsc Kesten}.} (see also~\cite[Remark~7.4]{Q})  to $H$, $K$ and $\nu$ provides us with $\epsilon >0$ and $\alpha>0$  such that the inequality $\nu^{n}(K^{\lfloor \epsilon n\rfloor}) \leq e^{-\alpha n}$ holds for all $n \geq 0$. Thus, the measure $\nu$ together with any open neighborhood $U \in \mathcal{U}(G)$ such that $U^{\left\lceil \epsilon^{-1}+1\right\rceil} \subseteq V$ witness the failure of the Kesten property for $G$, which completes the proof. \end{proof}

Beyond the realm of locally compact groups, the connection between amenability and the (strong) Kesten property appears to be more subtle, and there is no universal implication in either direction, as explained by the subsequent observations.

\begin{corollary}\label{corollary:kesten.lc} Every discrete topological subgroup of a topological group with the strong Kesten property is amenable. \end{corollary}

\begin{proof} Let $G$ be a topological group with the strong Kesten property. If $H$ is a discrete subgroup of $G$, then $H$ has the strong Kesten property by Remark~\ref{remark:subgroups} and thus must be amenable by Proposition~\ref{proposition:kesten.lc}. \end{proof}

\begin{corollary}\label{corollary:kesten.lc.free} If a topological group contains a discrete topological subgroup isomorphic to $F_{2}$, then it does not have the strong Kesten property. \end{corollary}

\begin{proof} This is a direct consequence of Corollary~\ref{corollary:kesten.lc} and non-amenability of $F_{2}$. \end{proof}

\begin{examples}\label{examples:non.kesten} Corollary~\ref{corollary:kesten.lc.free} implies that there are many (even extremely) amenable topological groups without the strong Kesten property. For instance, \begin{itemize}
    \item[$(1)$] the unitary group $\U(H)$ of unitary operators on an infinite-dimensional separable Hilbert space $H$, equipped with the strong operator topology, is extremely amenable by a result of Gromov and Milman~\cite{GM}, and $F_{2}$ is isomorphic to a discrete topological subgroup of $\U(F_{2}) \cong \U(H)$ via its left regular representation,
    \item[$(2)$] the group $\Aut(\mathbb{Q},<)$ of order-preserving automorphisms of the naturally ordered set of rational numbers, endowed with the topology of point-wise convergence induced by the discrete topology on $\mathbb{Q}$, is extremely amenable by work of Pestov~\cite{P_98} and, moreover, contains a discrete topological subgroup isomorphic to $F_{2}$ (see, e.g.,~\cite[Proof of Proposition~9.6]{JS} for details),
    \item[$(3)$] by work of Carderi and Thom~\cite{CT}, the metric completion of the inductive limit of diagonally embedded special linear groups over a finite field $F$, with respect to the joint extension of the normalized rank distances, is extremely amenable and, if $\cha(F) \ne 2$, contains a discrete topological subgroup isomorphic~to~$F_{2}$.
\end{itemize} \end{examples}

For further discussion, let us agree on the following convenient terminology.

\begin{definition} A topological group $G$ will be called \begin{itemize}
    \item[---\,] \emph{locally generated} if, for every $U \in \mathcal{U}(G)$, the subgroup of $G$ generated by $U$ coincides with $G$ (or equivalently, if $G$ does not have any proper open subgroup),
    \item[---\,] \emph{boundedly locally generated} if, for every $U \in \mathcal{U}(G)$, there exists a natural number $n$ such that $G=U^{n}$.
\end{itemize} \end{definition}

Note that every connected topological group is locally generated and that non-trivial non-archimedean groups are never locally generated.

\begin{remark}\label{remark:blg.implies.kesten} Every boundedly locally generated topological group has the Kesten property. \end{remark}

In general, the Kesten property does not imply amenability, even for topological groups with SIN. In view of Remark~\ref{remark:blg.implies.kesten}, this is witnessed by the following examples of non-amenable, boundedly locally generated topological groups with SIN.

\begin{examples}\label{examples:groups} \begin{enumerate}
    \item[$(1)$] The group $L^{0}(\mu,G)$ of all $\mu$-equivalence classes of measurable maps from a non-atomic standard probability space $(X,\mu)$ to a Polish group $G$, endowed with the topology of convergence in measure, is a Polish group~\cite[Proposition~7]{Moore}, has SIN if $G$ has SIN, is boundedly locally generated (see~\cite[Lemma~3.8]{SchneiderSolecki} for an even stronger property), and is non-amenable if $G$ is non-amenable~\cite[Theorem~1.1]{PS}. So, the discrete group $G = F_{2}$ would be a suitable choice to produce an example of the kind requested.
    \item[$(2)$] The full group $[R]$ of a countable Borel equivalence relation $R$ equipped with the uniform distance corresponding to a non-atomic invariant ergodic measure is a Polish topological group with SIN~\cite{K}, which is boundedly locally generated (as follows from~\cite[Theorem 3.3]{K}), and non-amenable if and only if $R$ is non-amenable~\cite{GP}.
    \item[$(3)$] The unit group $\GL(R)$ of a non-discrete irreducible, continuous ring $R$ (in von Neumann's sense~\cite{VonNeumannBook}), furnished with the topology induced by the rank metric, is a topological group with SIN, which is Polish if $R$ is separable, and moreover boundedly locally generated (see~\cite[Theorem~6.4]{Schneider25} for an even stronger property). For the concrete example of $R$ being the ring of densely defined, closed linear operators affiliated with the group von Neumann algebra of $F_{2}$, the topological group $\GL(R)$ is known to be non-amenable~\cite{SchneiderIMRN}.
\end{enumerate} \end{examples} 

In the other direction, we obtain the following result, using a characterization of amenability for topological groups given in~\cite[Theorem~5.3]{ST_F}.

\begin{theorem}\label{SINKesten} Every amenable topological group with SIN has the Kesten property. \end{theorem}

\begin{proof} Let $G$ be an amenable topological group with SIN. Consider any symmetric probability measure $\nu$ on $G$ with countable support. Let $U \in \mathcal{U}(G)$ be open. Since $G$ has SIN, we may assume that $U$ is invariant under conjugation and inversion. According to~\cite[Theorem~5.3, (1)$\Leftrightarrow$(2)]{ST_F}, there exists $\alpha \colon G \to \Sym (G)$ such that the action of the group generated by $\alpha(G)$ on the set $G$ is amenable\footnote{This refers to the usual definition of amenability for group actions on sets.} and \begin{equation}\tag{$\ast$}\label{eq:approximation}
    \forall g,h \in G \colon \quad \alpha(g)(h) \in Ugh .
\end{equation} For any $g,h \in G$, it follows that \begin{displaymath}
    h \, = \, \alpha(g)(\alpha(g)^{-1}(h)) \, \stackrel{\eqref{eq:approximation}}{\in} \, Ug\alpha(g)^{-1}(h)
\end{displaymath} and thus, by symmetry and conjugation-invariance of $U$, \begin{displaymath}
    \alpha(g)^{-1}(h) \, \in \, Ug^{-1}h .
\end{displaymath} This shows that \begin{equation}\tag{$\ast\ast$}\label{eq:approximation.2}
    \forall \epsilon \in \{ 1,-1 \} \ \forall g,h \in G \colon \quad \alpha(g)^{\epsilon}(h) \in Ug^{\epsilon}h .
\end{equation}

In order to obtain a lower bound on $\limsup_{n\to \infty}\nu^{n}(U^n)$ we are going to compare the random walk on $G$ driven by $\nu$ with a suitable random walk on $\Sym(G)$ defined below. For this purpose, first  we choose a set $S \subseteq G$ such that $S \cap S^{-1} =\{g \in S \mid g=g^{-1}\}$ and $S \cup S^{-1}=\supp(\nu)$. Next, we consider the formal disjoint union \begin{displaymath}
    S \sqcup S^{-1} \, := \, \{ (0,s) \mid s \in S \} \cupdot \{ (1,s) \mid s \in S^{-1} \} .
\end{displaymath} Note that, for any $s \in S \cup S^{-1}$, \begin{displaymath}
    \{ (0,s), (1,s) \} \subseteq S \sqcup S^{-1}\! \quad \Longleftrightarrow \quad s = s^{-1} .
\end{displaymath} Then $\nu$ naturally induces a probability measure $\nu_1$ on $S \sqcup S^{-1}$ defined as \begin{displaymath}
    \nu_1(i,s) \, := \, \begin{cases}
        \, \tfrac{\nu(s)}{2} & \text{if } s = s^{-1}, \\
        \, \nu(s) & \text{otherwise}.
    \end{cases}
\end{displaymath} Considering the map $\pi \colon S \sqcup S^{-1} \to G, \, (i,s) \mapsto s$, we notice that $\pi_{\ast}(\nu_{1}) = \nu$. Now, let $\Gamma$ be the subgroup of $\Sym(G)$ generated by $\alpha(S)\cup \alpha(S)^{-1}$. For $t=(i,s) \in S \sqcup S^{-1}$, we define \begin{displaymath}
    \Gamma \, \ni \, \alpha_{t} \, := \, \begin{cases}
        \, \alpha(s) & \text{if } i=0, \\
        \, \alpha(s^{-1})^{-1} & \text{if } i=1 .
    \end{cases}
\end{displaymath} We obtain a symmetric probability measure $\nu'$ on $\Gamma$ by setting  \begin{displaymath}
    \left. \nu'(h) \, := \, \nu_1\!\left( \left\{t \in S \sqcup S^{-1} \, \right\vert \alpha_t=h \right\}\right) \qquad (h \in \Gamma).
\end{displaymath} Equivalently, $\nu'$ is the push-forward measure of $\nu_{1}$ along the map \begin{displaymath}
    S \sqcup S^{-1}\! \, \longrightarrow \, \Gamma, \quad t \, \longmapsto \, \alpha_{t} .
\end{displaymath}
 
Now we are ready to show that \begin{equation}\tag{$\ast\!\ast\! \ast$}\label{eq:lower.bound}
    \forall n \in \mathbb{N} \colon \qquad \sup\nolimits_{x \in G} (\nu')^{n}(\{ g \in \Gamma \mid gx = x\}) \, \leq \, \nu^{n}(U^n) .
\end{equation} To this end, let $n \in \mathbb{N}$ and $x \in G$. Define \begin{align*}
    R_n \, &:= \, \{ (t_n,\ldots,t_1) \in (S\sqcup S^{-1})^n \mid \pi(t_n) \cdots \pi(t_1)  \in U^n \} , \\
    R_n'(x) \, &:= \, \{ (t_n,\ldots,t_1) \in (S\sqcup S^{-1})^n \mid (\alpha_{t_n} \circ \ldots \circ \alpha_{t_1})(x) = x \} .
\end{align*} Note that, if $(t_n,\ldots,t_1) \in R_n'(x)$, then $(\alpha_{t_n} \circ \ldots \circ \alpha_{t_1})(x) = x$ and so by~\eqref{eq:approximation.2} there exist $u_1,\ldots, u_n \in U$ such that \begin{displaymath}
    u_n\pi(t_n)u_{n-1}\pi(t_{n-1})\cdots u_1\pi(t_1) x\, = \, x,
\end{displaymath} which, upon multiplying by $x^{-1}$ from the right, yields that \begin{displaymath}
    u_n\pi(t_n)u_{n-1}\pi(t_{n-1})\cdots u_1\pi(t_1) \, = \, 1,
\end{displaymath} thus, thanks to invariance of $U$ under inversion and conjugation, \begin{displaymath}
    \pi(t_n)\cdots \pi(t_1) \, = \, \prod\nolimits_{i=1}^{n}(u_i^{-1})^{\pi(t_{i+1}) \cdots \pi(t_n)} \, \in \, U^n
\end{displaymath} and therefore $(t_n,\ldots,t_1) \in R_n$. This shows that $R_{n}'(x) \subseteq R_{n}$. Combining this with the definitions of $\nu_{1}$ and $\nu'$ as well as the fact that $\pi_{\ast}(\nu_{1}) = \nu$, we see that \begin{align*}
    (\nu')^{n}(\{ g \in \Gamma \mid gx = x\}) \, &= \, \sum\nolimits_{ (t_n,\ldots,t_1)\in R_n'(x)} \nu_1(t_n) \cdots \nu_1(t_1) \\
    &\leq \, \sum\nolimits_{ (t_n,\ldots,t_1)\in R_n} \nu_1(t_n) \cdots \nu_1(t_1) \, = \, \nu^{n}(U^n) .
\end{align*} This proves~\eqref{eq:lower.bound}. Based on this, it remains to establish the desired lower bound on \begin{displaymath}
    \limsup\nolimits_{n \rightarrow \infty}\left(\sup\nolimits_{x \in G} (\nu')^n(\{g \in \Gamma \mid gx=x\})\right)^{1/n},
\end{displaymath} as we will do using the amenability of the action of $\Gamma$ on $G$.
        
Since the action of $\Gamma$ on $G$ is amenable, for any $\epsilon >0$ and any finite subset $E \subseteq \Gamma$, we find an \emph{$(E, \epsilon)$-F\o lner set}, i.e., a finite non-empty subset $F \subseteq G$ such that \begin{displaymath}
    \sum\nolimits_{g \in E}\lvert gF \triangle F\rvert \, \leq \, \epsilon \lvert F \rvert .
\end{displaymath} It is easy to see that, if $F \subseteq G$ is an $(E, \epsilon)$-F\o lner, then the intersection of $F$ with at least one of the orbits of the action of $\Gamma$ on $G$ must also be an $(E, \epsilon)$-Følner set. Indeed, let $F_1, \ldots, F_k$ be the full list of the intersections of $F$ with the orbits of the action of $\Gamma$. Then, the sets $F_1, \ldots, F_k$ partition $F$, and for any $g \in E$ and $i \neq j$ we have $gF_i \cap F_j = \emptyset$. Therefore, since $F$ is an $(E, \epsilon)$-Følner set,  \begin{displaymath}
    \sum\nolimits_{g \in E}|gF \triangle F| \, = \, \sum\nolimits_{i=1}^{k} \sum\nolimits_{g \in E}|gF_i \triangle F_i| \, \leq \, \epsilon |F| \, = \, \epsilon \sum\nolimits_{i=1}^{k}|F_i| .
\end{displaymath} But this implies that for some $i \leq k$ we have
\[ \sum\nolimits_{g \in E}|gF_i \triangle F_i| \, \leq \, \epsilon |F_i|,\]
so $F_i$ is an $(E, \epsilon)$-F\o lner set as well.
Thus, we may always assume that the points  in the $(E, \epsilon)$-Følner set belong to one orbit of the action of $\Gamma$ on $G$ (but the orbit may depend on the choice of $E$ and $\epsilon$).

Now, let us fix some $\epsilon >0$ and choose a finite symmetric subset $E \subseteq \Gamma$ such that $\nu'(\Gamma \setminus E) \leq \epsilon/2 $. Let $F \subseteq G$ be an $(E, \epsilon/2)$-Følner set contained in some orbit $O$ of the action of $\Gamma$ on $G$. The $\nu'$-random walk on $O$ naturally induces a connected network on $O$ where the conductances are given by the transition probabilities.\footnote{We follow the terminology of \cite[Section 6.1 and 6.2]{LP}. The edge boundary and isoperimetric constant are defined in~\cite[Section 6.1, p.~177]{LP}. In that text, Mohar's inequality is stated for infinite networks, but it is also true for finite networks for trivial reasons.} In this network, the edge boundary of $F$ has the total weight at most $|F|(\epsilon/2+\epsilon/2)=|F|\epsilon$. Hence, the edge-isoperimetric ratio\footnote{In a network corresponding to a random walk, the edge-isoperimetric ratio of a finite set $F$ is equal to the weight of its edge boundary divided by $|F|$.}  of $F$ does not exceed $\epsilon$. As a result, the isoperimetric constant of the $\nu'$-random walk on $O$ does not exceed $\epsilon.$ Since $\epsilon$ was an arbitrary positive number, we conclude that
the infimum of isoperimetric constants of $\nu'$-random walks on $\Gamma$-orbits on $G$ is equal to $0$. Then Mohar's isoperimetric inequality (see Theorem~\ref{Mohar} or~\cite[Theorem~6.7]{LP}) implies that the supremum of the spectral radii of $\nu'$-random walks on $\Gamma$-orbits on $G$ is equal to $1$, hence \begin{displaymath}
    \limsup\nolimits_{n \rightarrow \infty}\left(\sup\nolimits_{x \in G} (\nu')^n(\{g \in \Gamma \mid gx=x\})\right)^{1/n} \, = \, 1,
\end{displaymath} which by~\eqref{eq:lower.bound} entails that \begin{displaymath}
    \limsup\nolimits_{n \rightarrow \infty} \nu^n (U^n)^{1/n} \, = \, 1. \qedhere
\end{displaymath} \end{proof}

One may wonder whether in Theorem~\ref{SINKesten} the SIN property may be dropped from the list of assumptions in exchange for the requirement of some other property, such as local generation. However, in Corollary~\ref{counterexample_K}, we provide an example of an amenable contractible Polish group without Kesten's property.
    
As mentioned in Remark~\ref{remark:blg.implies.kesten}, the Kesten property is a consequence of bounded local generation. We note that there exist topological groups that satisfy the hypotheses of Theorem~\ref{SINKesten}, but are not boundedly locally generated. 

\begin{example}[{\cite[Example~3.5]{P_17}}] The group $\U(\infty)_2$ of all unitary operators on $H = \ell^2(\mathbb{N})$ of the form $\mathrm{Id}_{H}+T$, where $T$ is a Hilbert--Schmidt operator on $H$, equipped with the topology defined by the Hilbert--Schmidt metric, constitutes a Polish topological group with SIN, which is even extremely amenable~\cite{GM} (see also~\cite[Corollary~4.1.13]{P_06}), but not bounded in the sense of~\cite{R} and thus not boundedly locally generated. \end{example}

For topological groups with SIN, amenability is equivalent to \emph{skew-amenability} as defined in~\cite{VP2020,JS}. The amenable topological groups without Kesten's property discussed in Example~\ref{examples:non.kesten}(1) and Example~\ref{examples:non.kesten}(3) are not skew-amenable, according to~\cite[example after
Proposition~3]{VP2020} and~\cite[Proposition~9.6]{JS}. This leads us to the following question.

\begin{question} Do all skew-amenable topological groups have the Kesten property? \end{question}

\section{Extensive amenability and inverted orbits}\label{section:extensive.amenability}

This section contains a brief overview of extensive amenability and inverted orbits as well as a discussion of questions connecting these topics with the amenability of equivalence relations. We do not claim any  significant new results in this section, and its main aim is to provide additional background, context, and motivation for the discussion in Section \ref{section:lamplighters}.

Extensive amenability is a strengthening of the notion of amenability of a group action. This property was formally defined in \cite{JMBMdlS}, but it was implicitly used as an important tool in the solutions of amenability problems of several classes of groups in \cite{JM} and in \cite{JNdlS}. Amenability problems for the group of the interval exchange transformations and for Thompson's group $F$ can also be reduced to a verification of extensive amenability of specific actions of these groups (see \cite{JMBMdlS} and \cite{JS}).
	
In order to give a definition of extensive amenability, we will use following auxiliary constructions. Let $G$ be a group $G$ acting on a set $X$. Denote by $\mathcal{P}_f(X)$ the group of all finite subsets of $X$ with the multiplication defined as the symmetric difference. It is easy to see that $\mathcal{P}_f(X)$ is isomorphic to the direct sum $\bigoplus_{x\in X} \mathbb{Z}_2$. The action of $G$ on $X$ naturally extends to an action of $G$ on $\mathcal{P}_f(X)$, and if we view $\mathcal{P}_f(X)$ as $\bigoplus_{x\in X} \mathbb{Z}_2$ the corresponding action can be defined as \begin{displaymath} 
    gw(x) \, := \, w(g^{-1}x) \qquad ( g \in G, \, w \in \mathcal{P}_f(X), \, x \in X ).
\end{displaymath} This action allows us to define the semidirect product $\mathcal{P}_f(X)\rtimes G$.  Moreover, the coset space $(\mathcal{P}_f(X)\rtimes G)/G$ is naturally isomorphic to $\mathcal{P}_f(X)$, and we call the action $\mathcal{P}_f(X)\rtimes G$ on $(\mathcal{P}_f(X)\rtimes G)/G$ the affine action of  $\mathcal{P}_f(X)\rtimes G$ on $\mathcal{P}_f(X)$. More explicitly, \begin{displaymath}
    (E,g)F \, = \, E \triangle gF \qquad (g \in G, \, E,F \in \mathcal{P}_f(X)). 
\end{displaymath}

\begin{definition}[\cite{JMBMdlS}] An action of a group $G$ on a set $X$  is called \emph{extensively amenable} if the induced affine action of $\mathcal{P}_f(X)\rtimes G$ on $\mathcal{P}_f(X)$ is amenable. \end{definition} 

We refer the reader to~\cite[Chapter~5]{J} and~\cite[Section~2]{JMBMdlS} for several equivalent definitions of extensive amenability.

It is not difficult to prove that extensive amenability of an action on a non-empty set implies its amenability. The definition of extensive amenability also implies that any action of an amenable group is extensively amenable. Moreover, any transitive action of a finitely generated group with a recurrent orbital Schreier graph is extensively amenable (see \cite[Theorem 1.2]{JNdlS}). Recurrent actions have been the main source of non-trivial examples of extensively amenable actions, however, it is still unknown if there is an action with  uniformly subexponential growth of orbital Schreier graphs which fails to be extensively amenable. 
	
Extensive amenability of actions is preserved under taking direct products (see \cite[Corollary~2.6]{JMBMdlS}), direct limits (\cite[Proposition~2.2]{JMBMdlS}), and extensions (see ~\cite[Proposition~ 2.4]{JMBMdlS}. Moreover, extensively amenable action is also hereditarily amenable (see~\cite[Corollary~2.3]{JMBMdlS}). However, the converse is not true, as shown in~\cite[Section~6]{JMBMdlS}.  

Let us describe a probabilistic characterization of extensive amenability in terms of the inverted orbits of group actions.
	
\begin{definition}\label{invorb.Def} Let $G$ be a group acting on a set $X$. For a sequence $h=(h_1,\ldots,h_n)$ of elements of $G$ and a point $x \in X$, the \emph{inverted orbit} of $x$ under $h$ is defined as the set $O_h(x) := \{x, h_nx, h_nh_{n-1}x, \ldots, h_nh_{n-1}\cdots h_1x \}$. \end{definition}

In practice, the sequence $h$ from the definition above often represents the sequence of increments of a random walk. The name ``inverted orbit'' is related to the following interpretation of these sets. If $h$ defines a sequence of positions $g_k$, $k=1,\ldots,n $, of a left random walk on $G$ given by $g_0=1$, $g_k=h_kg_{k-1}$, $k=1,\ldots,n $, then the inverted orbit of a point $x \in X$ can be defined as the set $\{x, g_1^{-1}x,\ldots, g_n^{-1}x\} $. If one replaces elements of $h$ by their inverses in Definition~\ref{invorb.Def}, the resulting set would be exactly the one described in the previous sentence. However, Definition~\ref{invorb.Def} is more convenient when one works with lamplighter groups, and thus we will use it. Moreover, when $h$ is sampled according to a symmetric probability measure on $G$, for every point $x \in X$  the distributions of the inverted orbits for these two definitions coincide.
	
As shown in \cite[Proposition 4.1]{JMBMdlS}, behavior of inverted orbits characterizes extensive amenability of actions. We restate part of this result for reader's convenience below. The notation is as follows. For a point $x\in X$ we will denote by $O_n(x)$ the inverted orbit of $x$ under the first $n$ steps (increments) of a symmetric random walk on $G$, so formally $O_n(x)$ is a random subset of $X$. Finally, the expectation and probability are computed with respect to the random walk on $G$.
	
\begin{proposition}[{\cite[Proposition 4.1]{JMBMdlS}}]\label{inv_orb_char} Let $G$ be a finitely generated group acting transitively on a set $X$, and let $\mu$ be a finitely supported symmetric non-degenerate probability measure $\nu$ on $G$. Consider corresponding random walks on $G$ and $X$. \begin{enumerate}
	\item[$(1)$] The action of $G$ on $X$ is extensively amenable if and only if, for every $x \in X$,  
			$$ \qquad \lim\nolimits_{n \rightarrow \infty} -\tfrac{1}{n}  \log \mathbb{E}2^{-|O_n(x)|} \, = \, 0 .$$
	\item[$(2)$] The action of $G$ on $X$ is extensively amenable if and only if, for every $ x \in X$ and every $\epsilon \in \mathbb{R}_{>0}$, $$ \qquad \limsup \nolimits_{n \rightarrow \infty} \mathbb{P}(|O_n(x)| \leq \epsilon n)^{1/n} \, = \, 1.$$
\end{enumerate} \end{proposition}

The idea behind the proof of~(1) given in~\cite{JMBMdlS} is to express $\mathbb{E}2^{-|O_n(x)|}$ as a return  probability for a random walk on $\mathcal{P}_f(X)$ induced by a switch-walk-switch random walk on $\mathcal{P}_f(X)\rtimes G$ and then apply Kesten's criterion for amenability of actions. The statement of~(2) is equivalent to the statement of~\cite[Proposition~4.1(2)]{JMBMdlS}.
	
\begin{remark}\label{remark:limit_for_inv_orbits} Let a countable group $G$ act on a set $X$.  Consider any symmetric random walk on $G$ and fix arbitrary $x \in X$. Then, for every $\epsilon >0$, the function $ n \mapsto \mathbb{P} (|O_n(x)| \leq  \epsilon n)$ is supermultiplicative, and the following limit exists 
	 $$p_{\epsilon}(x) \, := \, \lim\nolimits_{n \rightarrow \infty} \mathbb{P}(|O_n(x)| \leq \epsilon n)^{1/n} .$$
Indeed, for any two sequences of increments $h \in G^{n}, t \in G^{m}$, let us denote by $(h,t)$ their concatenation. Then, for  every point $x \in X$, we have $$|O_{(h,t)}(x)| \, = \, |O_{t}(x) \cup t_m\cdots t_1O_{h}(x) | \, \leq \, |O_{h}(x)|+ |O_{t}(x)|.$$ As a result, for any symmetric random walk on $G$ and any $m,n \in \mathbb{N}$, we have \begin{displaymath}
    \mathbb{P}(|O_{n+m}(x)| \leq \epsilon (n+m)) \, \geq \, \mathbb{P}(|O_n(x)| \leq \epsilon n)\mathbb{P}(|O_m(x)| \leq \epsilon m) .
\end{displaymath} Hence, the function $ n \mapsto \mathbb{P} (|O_n(x)| \leq  \epsilon n)$ is supermultiplicative, so Fekete's lemma implies that the limit $p_{\epsilon}(x)= \lim\nolimits_{n \rightarrow \infty } {\mathbb{P} (|O_n(x)| \leq  \epsilon n)}^{1/n} $ exists.  Moreover, we also have $p_{\epsilon}(x)= \sup_{n \in \mathbb{N}}{\mathbb{P} (|O_n(x)| \leq  \epsilon n)}^{1/n}$. In particular, this implies that $$\mathbb{P} (|O_n(x)| \leq \epsilon n) \, \leq \, p_{\epsilon}(x)^n$$ for every $n\in \mathbb{N}$. \end{remark} 

One may wonder if amenability of an equivalence relation generated by a measure-preserving action of a group on a standard probability space has implications for extensive amenability of the actions of this group on the orbits of individual points.
 
In particular, an affirmative answer to the following question would provide the first non-elementary family of extensively amenable actions of groups with non-recurrent orbital Schreier graphs. It would also bring us close to proving amenability of the $\IET$ group (see~\cite[Lemma~2.2 and Proposition~5.3]{JMBMdlS}).  

\begin{question}\label{ExtAm} Let $G$ be a countable group acting on a standard measure space $(X,\mu)$ in a Borel way such that $\mu$ is invariant and the induced orbit equivalence relation is $\mu$-amenable. Is the action of $G$ on almost every orbit extensively amenable? \end{question}

Even though an affirmative answer would imply that actions on the orbits exhibit a rather strong property, there is some evidence supporting this conjecture, and we list several relevant facts below.

\begin{enumerate}
    \item A variety of works studying invariant random subgroups (IRS) of discrete groups show that for an ergodic probability measure preserving action of a group $G$ on a standard measure space $(X, \mu)$, for almost every $x \in X$, the random walks on the orbital Schreier graph of $x$ often have properties similar to the random walks on the Cayley graphs of quotients of $G$ modulo its normal subgroups. For example, see \cite[Theorem~3]{AGV} for the results concerning the spectral radius, and \cite[Theorem~3.1]{B14} for the definition of the asymptotic (average) entropy in this setting. Since extensive amenability is characterized in terms of random walks, a positive answer to Question \ref{ExtAm} appears possible.
    \item It is well-known that in the setting of~(1), amenability of the orbit equivalence relation generated by $G$ implies that for almost every $x \in X$, the action of $G$ on the orbit of $x$ has properties stronger than amenability, and in particular, it is hereditarily amenable, see for example \cite{Kaimanovich97} or \cite[Chapter~9]{KM}.
    \item The operations preserving amenability of equivalence relations with respect to measures, such as taking direct products or directed unions, also preserve extensive amenability of group actions, see~\cite[Section~2]{JMBMdlS}.
 \end{enumerate}

One may also wonder if the conditions from Proposition~\ref{inv_orb_char} hold on average for a group action on a standard measure space with amenable orbit equivalence relation. To simplify the notation, let us say that a sequence $(a_n)_{n \geq 1}$ in $[0,1]$ \emph{decays subexponentially} if \begin{displaymath}
    \limsup\nolimits_{n \rightarrow \infty} a_{n}^{1/n}\, = \, 1 .
\end{displaymath}

\begin{question}\label{large_avareges} Let $R$ be a countable Borel equivalence relation on a standard measure space $(X,\mu)$, such that $\mu$ is ergodic and invariant with respect to $R$ and $R$ is $\mu$-amenable. Let $G$ be a countable subgroup of $[R]$, and consider a symmetric probability measure $\nu$ on $G$. For $x \in X$, let $O_n(x)$ denote the inverted orbit of $x$ under the first $n$ steps of the random walk on $G$ driven by $\nu$. Under these conditions, does $\int_X\mathbb{E}2^{-|O_n(x)|} \,\mathrm{d}\mu(x)$ decay subexponentially, and does $\int_X \mathbb{P}(|O_n(x)| \leq \epsilon n) \,\mathrm{d}\mu(x)$ decay subexponentially for every $\epsilon > 0$? \end{question}
	
\section{Measurable lamplighter groups} \label{section:lamplighters}
	
In this section, we construct an analogue of the lamplighter group in the measurable setting and show how Kesten's property for this group is related to the behavior of the inverted orbits of points for the actions of discrete groups. In particular, Theorem~\ref{inv.orb.T}  and Theorem~\ref{Top.Lamplighter} allow us to construct an example of an amenable, locally generated (in fact, even contractible) Polish group without Kesten's property (Corollary~\ref{counterexample_K}). On the other hand, it is unknown if Kesten's property holds for the measurable lamplighter corresponding to the ergodic amenable probability measure preserving relation, and if it holds, then  
 Theorem~\ref{inv.orb.T} implies that  Question~\ref{large_avareges} has an affirmative answer.

Let $R$ be a countable Borel equivalence relation on standard measure space $(X,\mu)$ such that $\mu$ is $R$-quasi-invariant. We equip $R$ with the Borel structure inherited from $X \times X$. Then the measure $\mu_{R}$ on the algebra of Borel subsets of $R$ is defined by \begin{displaymath}
    \mu_{R}(A) \, := \, \int_X|A_x|\,\mathrm{d}\mu(x) ,
\end{displaymath}  where \begin{displaymath}
    A_x \, := \, \{y \in X \mid (x,y) \in A\}.
\end{displaymath} One readily checks that the set of all Borel subsets $A \subseteq R$ with $\mu_{R}(A) < \infty$ constitutes an abelian group with respect to symmetric difference. We consider the quotient group \begin{displaymath}
    \A_{\mu}(R) \, := \, \{ [A]_{\mu_{R}} \mid A \subseteq R \text{ Borel}, \, \mu_{R}(A) < \infty \}
\end{displaymath} of the corresponding $\mu_{R}$-equivalence classes.\footnote{As is customary, we will not distinguish in notation between sets and their equivalence classes.} To simplify the notation, we will denote the group operation on $\A_{\mu}(R)$ by ``$+$'', and the identity in this group is denoted by $\emptyset$.
Note that $\A_{\mu}(R)$ admits a natural embedding into $L^{1}(R,\mu_{R})$ sending an element of $\A_{\mu}(R)$ to the $\mu_{R}$-equivalence class of the indicator function of any concrete representative. We equip $\A_{\mu}(R)$ with the topology inherited from $(L^{1}(R,\mu_{R}),\Vert \cdot \Vert_{1})$ via this embedding, which we call the \emph{$L^{1}$-topology}. To be completely explicit, the $L^{1}$-topology on $\A_{\mu}(R)$ is generated by the metric \begin{displaymath}
    \A_{\mu}(R) \times \A_{\mu}(R) \, \longrightarrow \, \mathbb{R}_{\geq 0}, \quad (A,B) \, \longmapsto \, \mu_{R}(A \triangle B) ,
\end{displaymath} which we will refer to as the \emph{$L^{1}$-distance} on $\A_{\mu}(R)$. Since the $L^{1}$-distance is translation-invariant, $\A_{\mu}(R)$ with the $L^{1}$-topology constitutes a topological group. Notice that the action of the Borel full group $[R]_B$ on $X \times X$ defined by 
$$ g(x,y) \, := \, (x,gy),$$ 
induces an action of $[R]$ on $\A_{\mu}(R)$ by isometries.\footnote{Since transformations in $[R]$ are defined modulo sets of measure $0$, we start the construction with an action of $[R]_B$ on $X \times X$. However, when we consider a countable subgroup $G$ of $[R]$ later, we may always pass to a well-defined action of $G$ on a subset $A \times A \subseteq X \times X$ where $A$ is an $R$-invariant subset of full measure.} Indeed, it is easy to see that this action preserves $R$ and induces a natural action of $[R]_B$ on the Borel subsets of $R$ which also preserves $\mu_{R}$. Next, if $c'$ and $c''$ are any representatives of $c \in \A_{\mu}(R)$ and $g \in [R]_B$, then \begin{displaymath}
    \mu_{R}(gc' \triangle gc'') \, = \, \mu_{R}(g(c' \triangle c'')) \, = \, \mu_{R}(c' \triangle c'') \, = \, 0,
\end{displaymath} so $gc'$ and $gc''$ represent the same element of $\A_{\mu}(R)$ as well. Moreover, if $g, h \in [R]_B$  represent the same element of $[R]$, then the set \[E \, := \, \{x \in X \mid g(x)\neq h(x)\}\] is a Borel subset of $X$ and $\mu(E)=0$. Since $\mu$ is $R$-quasi-invariant, this implies that
\[\mu \left(\{x \in X \mid  R_x \cap E \neq \emptyset \}\right) \, = \, 0.\]  Let $c' \subseteq R$ be any representative of $c \in \A_{\mu}(R)$.  The discussion above implies that
\[\mu \left(\{x \in X \mid  c'_x \cap E \neq \emptyset \}\right) \, = \, 0.\]
However, for any $x \in X$ the condition $c'_x \cap E =\emptyset$ implies that $gc'_x=hc'_x$.  Thus, we have
\[\mu_{R}(gc'\triangle hc') \, = \, \int_X |gc'_x \triangle hc'_x | \, \mathrm{d}\mu(x) \, = \, 0.\]
Therefore, if $g$ and $h$ represent the same element of $[R]$, then $gc'$ and $hc'$ represent the same element of $\A_{\mu}(R)$. Hence, the action of $[R]_B$ on $X \times X$ induces a well-defined action of $[R]$, viewed as an abstract group, on $\A_{\mu}(R)$. The fact that this action preserves $L^{1}$-distance on $\A_{\mu}(R)$ follows immediately from the definitions.

The next theorem shows that the semidirect product induced by the action described above has a natural topological structure.

\begin{theorem} \label{Top.Lamplighter} Let $R$ be a countable Borel equivalence relation on a standard measure space $(X,\mu)$ such that $\mu$ is $R$-quasi-invariant. Let $[R]$ be endowed with the uniform topology and let $\A_{\mu}(R)$ be endowed with the $L^{1}$-topology. Then $\A_{\mu}(R) \rtimes [R]$ with the product topology is a Polish topological group. Moreover, the following hold. \begin{enumerate}
	\item[$(1)$] If $\mu$ is $R$-invariant, then $\A_{\mu}(R) \rtimes [R]$ admits a complete left-invariant metric compatible with its topology, i.e., $\A_{\mu}(R) \rtimes [R]$ is a CLI group\footnote{In the sense of~\cite[Remark after Proposition~3.5, p.~252]{LeMaitre18}}.
    \item[$(2)$] If $R$ is $\mu$-amenable, then the topological group $\A_{\mu}(R) \rtimes[R]$ is amenable.
    \item[$(3)$] If $\mu$ is $R$-ergodic, then $\A_{\mu}(R) \rtimes[R]$ is contractible and does not have SIN.
\end{enumerate}	\end{theorem} 

\begin{proof} Let $d_{[R]}$ denote the uniform distance on $[R]$ and $d_{\A_{\mu}(R)}$ denote the $L^{1}$-distance on $\A_{\mu}(R)$. It is easy to see that that the metric $d$ on $\A_{\mu}(R) \rtimes[R]$ defined by \begin{displaymath}
    d((c_1,g_1), (c_2,g_2)) \, := \, d_{\A_{\mu}(R)}(c_1,c_2)+d_{[R]}(g_1,g_2)
\end{displaymath} generates the product topology on $\A_{\mu}(R) \times [R]$. Since $[R]$ and $\A_{\mu}(R)$ are topological groups, in order to show that $\A_{\mu}(R) \rtimes[R]$ is a topological group we need to verify that the action \begin{displaymath}
    [R] \times \A_{\mu}(R) \, \longrightarrow \, \A_{\mu}(R), \quad (g,c) \, \longmapsto \, gc
\end{displaymath} is continuous (see, e.g.,~\cite[III, \S2.10, Proposition~28, p.~241]{bourbaki}). As $[R]$ acts by isometries on $(\A_{\mu}(R),d_{\A_{\mu}(R)})$, we see that \begin{displaymath}
    d_{\A_{\mu}(R)}(gc,g'c') \leq d_{\A_{\mu}(R)}(gc,g'c)+d_{\A_{\mu}(R)}(g'c,g'c') = d_{\A_{\mu}(R)}(gc,g'c)+d_{\A_{\mu}(R)}(c,c')
\end{displaymath} for all $g,g' \in [R]$ and any $c,c' \in \A_{\mu}(R)$. Thus, it suffices to check that, for any fixed $c \in \A_{\mu}(R)$, the map \begin{displaymath}
    [R] \, \longrightarrow \, \A_{\mu}(R), \quad g \, \longmapsto \, gc
\end{displaymath} is continuous. To this end, fix any $c \in \A_{\mu}(R)$ and $\epsilon_0 >0$. We need to find $\delta>0$ such that for any $g \in [R]$ with $\mu(\supp(g))<\delta$  we have $d_{\A_{\mu}(R)}(c,gc)< \epsilon_0$. 
Recall that for every $g \in [R]$  the definition of $d_{\A_{\mu}(R)}$ implies that
\[d_{\A_{\mu}(R)}(c,gc) \, = \, \int_X |c_x \triangle gc_x| \, \mathrm{d}\mu(x).\]
Notice that, for any $g \in [R]$, we have $|c_x\triangle gc_x| \leq 2|c_x| $ for $\mu$-almost every $x \in X$. Moreover, since $c \in \A_{\mu}(R)$, the function $X \to [0,\infty], \, x \mapsto 2|c_x|$ is $\mu$-integrable. Then, the absolute continuity of the Lebesgue integral implies that there exists $\epsilon >0$ such that, for any Borel $B \subseteq X$ with $\mu(B) < \epsilon$ and any $g \in [R]$, we have \[ \int_B |c_x\triangle gc_x| \, \mathrm{d}\mu(x) \, \leq \, \int_B 2|c_x| \, \mathrm{d}\mu(x) \, < \, \epsilon_0.\]
Let us fix this $\epsilon >0$. By the Feldman--Moore theorem,  see for example, \cite[Theorem 1.3]{KM}, we may assume that $R$ is generated by an action of a countable group $G$. Let $g_1,g_2,g_3,\ldots$ be some enumeration of its elements. Then, there exists $k \in  \mathbb{N}$ such that \begin{displaymath}
    \mu(\{x \in X \mid c_x \subseteq \{g_1(x),\ldots, g_k(x)\}\}) \, > \, 1-\epsilon/2 .
\end{displaymath} Notice that the measure $g_1 \mu + \ldots +g_k \mu $ is absolutely continuous with respect to $\mu$, and thus there exists $\delta >0 $ such that for any Borel $A \subseteq X$ with $\mu(A)< \delta$ we have $(g_1 \mu + \ldots +g_k \mu) (A) <\epsilon/2$. Furthermore, for any $g \in [R]$, the set $\{x \in X \mid c_x \neq gc_x\}$ is contained in \begin{displaymath}
    \{x \in X \mid c_x \nsubseteq \{g_1(x),\ldots, g_k(x)\}\} \cup \{x \in X \mid  \supp(g) \cap \{g_1(x),\ldots, g_k(x)\} \neq \emptyset \} .
\end{displaymath} Thus, if  $\mu(\supp(g)) < \delta$, then \begin{align*}
    &\mu(\{x \in X \mid c_x \neq gc_x\}) \\
    &\qquad \leq \, \mu(\{x \in X \mid c_x \nsubseteq \{g_1(x),\ldots, g_k(x)\}\}) +(g_1 \mu + \ldots + g_k \mu)(\supp(g)) \, < \, \epsilon. 
\end{align*} Choose $\delta >0$ as in the argument above, and consider any $g \in [R]$ with $\mu(\supp(g)) < \delta$. Let $B := \{x \in X \mid c_x \neq gc_x\}$, then we immediately get $\mu(B) < \epsilon$. But now the choice of $\epsilon >0 $ leads to the desired inequality \[d_{\A_{\mu}(R)}(c,gc) \, = \, \int_B|c_x\triangle gc_x| \, \mathrm{d}\mu(x) \, \leq \, \int_B 2|c_x| \, \mathrm{d}\mu(x) \, < \, \epsilon_0. \] Therefore,  $gc \rightarrow c$ as $g \rightarrow 1$ in $[R]$, and the map $g \mapsto gc$ from $[R]$ to $\A_{\mu}(R)$ is indeed continuous for any $c \in \A_{\mu}(R)$. 

Finally, we show that $\A_{\mu}(R) \rtimes [R]$ is Polish. First, the map sending any element $c$ of the group $\A_{\mu}(R)$ to the element of $L^1(R,\mu_{R})$ represented by the indicator function of any representative of $c$ is an isometry between $(\A_{\mu}(R), d_{\A_{\mu}(R)})$ and a closed subset of $(L^1 (R,\mu_{R}),\Vert \cdot \Vert_{1})$, namely the subset of all elements of $L^1 (R,\mu_{R})$ with an essential range in the closed set $\{0,1\}$. Thus, $(\A_{\mu}(R),d_{\A_{\mu}(R)})$ is a separable complete metric space. In particular, the $L^{1}$-topology on $\A_{\mu}(R)$ is Polish. Moreover, the uniform topology on $[R]$ is Polish as well~\cite[Section 5.1, p.~299]{GP}, whence the product topology on $\A_{\mu}(R) \rtimes [R]$ is Polish, too.

(1) The topology on $\A_{\mu}(R)\rtimes [R]$ is generated by the metric $d$, and one readily checks that $d$ is left-invariant. Moreover, as seen above, $(\A_{\mu}(R),d_{\A_{\mu}(R)})$ is a complete metric space. If $\mu$ is $R$-invariant, then $([R],d_{[R]})$ is also a complete metric space (see, e.g.,~\cite[Section~2, p.~921]{D95}), thus $(\A_{\mu}(R)\rtimes[R],d)$ is a complete metric space, too.

(2) This follows immediately from the previous part and the fact that amenability of topological groups is preserved under extensions, see~\cite[Theorem~4.8]{R67}. Indeed, one only needs to verify that $[R]$ is amenable even if $\mu$ is not ergodic. We will provide a standard argument for reader's convenience. Since $R$ is $\mu$-hyperfinite (see~\cite{CFW81}), on an $R$-invariant set of full measure, $R$ may be represented as the union of the increasing sequence of countable Borel equivalence relations $R_n, \, n \geq 1$, such that for every $n \geq 1$ the size of the classes of $R_n$ is bounded by a constant $K_n$.  Recall that a direct product of a family of symmetric groups of a uniformly bounded size is always a locally finite group, and hence it is amenable. Therefore, for each $n \geq 1$, the full group $[R_n]$ is amenable as a discrete group. Since the union of these groups is dense in $[R]$ equipped with the uniform topology, the topological group $[R]$ is amenable.
   
(3) We first prove that $\A_{\mu}(R) \rtimes[R]$ does not have SIN. Fix some small $r>0$ and consider an element $g \in [R]$ such that $\mu(\supp(g))<r$. Since $R$ is ergodic, there exists $c \in \A_{\mu}(R)$ such that $|c_x|=1$, and $c_x \subseteq \supp(g)$ for a.e.~$x \in X$. Then \[d_{\A_{\mu}(R)}(c, gc) \, = \, d_{\A_{\mu}(R)}(c+gc, \emptyset) \, = \, 2. \] Notice that if we conjugate the neighborhood $U=B_d(r)$ (the ball of radius $r$ centered at the identity of $\A_{\mu}(R)\rtimes[R]$) by $h=(c, 1)$, the only elements of $hUh^{-1}$ with the $[R]$-component equal to $g$ are the conjugates of the elements of the form $(s,g) \in U$. However, for any  $(s,g) \in U$, \[h(s,g)h^{-1} \, = \, (c, 1)(s,g)(c,1) \, = \, (c+gc+s,g),\] where $d_{\A_{\mu}(R)}(s, \emptyset) \leq r$. As a result, \[d_{\A_{\mu}(R)}(c+gc+s, \emptyset) \, \geq \, 2-r \, > \, r,\] so $h(s,g)h^{-1}$ does not belong to $U$ for any admissible $s$. Hence, the intersection $U \cap hUh^{-1}$ does not project onto the ball of radius $r$ in $[R]$. Repeating this construction for a sequence of $r_n \rightarrow 0$ shows that the intersection of all conjugates of $U$ does not project onto any ball around the identity in $[R]$, so it is not open in $\A_{\mu}(R) \rtimes[R]$.
    
Next, we show that $\A_{\mu}(R) \rtimes [R]$ is contractible. Since $(X,\mu)$ is a standard measure space and $\mu$ is $R$-ergodic, the full group $[R]$ equipped with the uniform topology is contractible by \cite[Section 2]{D95}. Thus, it suffices to show that $\A_{\mu}(R)$ is contractible, and we will provide an explicit construction of a homotopy between the identity map on $\A_{\mu}(R)$ and the constant mapping $c \mapsto \emptyset$. By the isomorphism theorem for standard measure spaces, see for example~\cite[Section 17.F]{K95}, we may assume that $(X, \mu) =([0,1], \lambda)$, where $\lambda$ is the Lebesgue measure restricted to $[0,1]$. Now, fix any $c\in \A_{\mu}(R)$, and for any $t \in [0,1]$ define $c(t) \in \A_{\mu}(R)$ as follows: choose any representative $c'$ of $c$, and let $c(t) \in \A_{\mu}(R)$ be the class of the set $c'(t) \subseteq R$ defined by \begin{displaymath}
    c'(t) \, := \, c' \cap ([0,t] \times [0,1]) .
\end{displaymath} It is easy to see that $c(t)$ is well-defined for every $t \in [0,1]$, and $c(0)= \emptyset$ and $c(1)=c$.\footnote{Recall that, by a mild abuse of notation, $\emptyset$ also denotes the identity of $\A_{\mu}(R)$.} Moreover, for any $s, t \in [0,1]$ with $s<t$, we have \begin{displaymath}
    d_{\A_{\mu}(R)}(c(s),c(t)) \, = \, \int_s^t|c_x| \, \mathrm{d}\lambda(x) .
\end{displaymath} Since $[0,1] \to [0,\infty], \, x \mapsto |c_x|$ is $\lambda$-integrable, the absolute continuity of the Lebesgue integral implies that the map $[0,1] \to \A_{\mu}(R), \, t \mapsto c(t)$ is continuous (so it is continuous path from $\emptyset$ and $c$ in $\A_{\mu}(R)$). Next, we will show that the map \begin{displaymath}
    h\colon \, \A_{\mu}(R) \times [0,1] \, \longrightarrow \, \A_{\mu}(R), \quad (c,t) \, \longmapsto \, c(t)
\end{displaymath} is continuous, and therefore, it is a homotopy witnessing contractibility of $\A_{\mu}(R)$. Notice that for any $t \in [0,1]$ and any $s,r \in \A_{\mu}(R)$, \[d_{\A_{\mu}(R)}(s(t),r(t)) \, = \, \int_0^t|s_x \triangle r_x | \, \mathrm{d}\lambda(x) \, \leq \, \int_0^1|s_x \triangle r_x | \, \mathrm{d}\lambda(x) \, = \, d_{\A_{\mu}(R)}(s,r).\]
Then, for any $t, t' \in [0,1]$ and any $s,r \in \A_{\mu}(R)$, \begin{align*}
    d_{\A_{\mu}(R)}(s(t'),r(t)) \, &\leq \, d_{\A_{\mu}(R)}(r(t), r(t'))+d(s(t'),r(t')) \\
    &\leq \, d_{\A_{\mu}(R)}(r(t),r(t'))+d_{\A_{\mu}(R)}(s,r).
\end{align*} Since $[0,1] \to \A_{\mu}(R), \, t \mapsto r(t)$ is continuous for every $r \in \A_{\mu}(R)$, we now conclude that $d_{\A_{\mu}(R)}(s(t'),r(t)) \rightarrow 0$ as $ (s,t') \rightarrow (r,t)$ in $\A_{\mu}(R) \times [0,1]$. So, $h$ is continuous. \end{proof}

\begin{question}\label{q: skew-lamplighter} Let $R$ be a countable Borel equivalence relation on a standard measure space $(X,\mu)$, and let $\mu$ be $R$-invariant and $R$-ergodic. Are the semidirect products $\A_{\mu}(R)\rtimes[R]$ and $\A_{\mu}(R)\rtimes G$ (for any countable dense subgroup $G \leq [R] $) always \emph{skew-amenable} in the sense of~\cite{VP2020,JS}?\footnote{According to Theorem~\ref{Top.Lamplighter}(3), the topological group $\A_{\mu}(R) \rtimes[R]$ does not have SIN, whence skew-amenability is not a consequence of amenability.} (In particular, it would be interesting to know whether $\A_{\mu}(R)\rtimes[R]$ and its versions with dense countable subgroups of $[R] $ can be \emph{proximally simulated}~\cite[Section 7]{JS} by suitable subgroups.) \end{question}
	
Next, we explain the connection between the (strong) Kesten property for measurable lamplighters and the behavior of the inverted orbits.

Let $R$ be a countable Borel equivalence relation on a standard measure space $(X,\mu)$ and let $\mu$ be $R$-quasi-invariant. Assume that $G$ is a countable subgroup of $[R]$ and let $\A_{\mu}(G)$ be the subgroup of $\A_{\mu}(R)$ generated by the graphs of the elements in $G$. Then $\A_{\mu}(G)$ is invariant under the action of $G$, and we can consider $\A_{\mu}(G) \rtimes G$ as a subgroup of $\A_{\mu}(R) \rtimes [R]$. In order to introduce some additional notation for the statement of our next theorem, let $\nu$ be a symmetric non-degenerate measure on $G$. Then the corresponding \textit{switch-walk-switch} measure $\hat{\nu}$ on $\A_{\mu}(G) \rtimes G$ is defined by \[\hat{\nu} \, := \, 1/2(\delta_{\emptyset}+\delta_{I})*\nu*1/2(\delta_{\emptyset}+\delta_{I}),\] where $\emptyset$ is viewed as the trivial element of $\A_{\mu}(R)$ and $I \subseteq R$ is the graph of the identity map. Finally, we denote by $$\hat{g}_n=(c_n,g_n)$$  the position at time $n$ of the left $\hat{\nu}$-random walk started at the identity.
   
\begin{theorem}\label{inv.orb.T} Let $R$ be an countable Borel equivalence relation on a standard measure space $(X,\mu)$ and let $\mu$ be $R$-quasi-invariant and $R$-ergodic. Let $d$ denote the $L^{1}$-distance on $\A_{\mu}(R)$. Consider a symmetric finitely supported probability measure $\nu$ on $[R]$. Let $G$ be the subgroup of $[R]$ generated by the support of $\nu$, and let $\hat{\nu}$ be the corresponding switch-walk-switch random walk on $\A_{\mu}(G) \rtimes G$. Finally, let $O_n(x)$ denote the inverted orbit of $x$ under the first $n$ steps of the $\nu$-random walk on $G$. Then the following statements hold. \begin{enumerate}
	\item[$(1)$]\label{strong} For every $\epsilon \in (0,1)$,we have\footnote{The expectation in the right-hand side is taken with respect to $\nu$-random walk on $G$, and the probability in the left-hand side is taken with respect to $\hat{\nu}$-random walk on the semidirect product.} \begin{displaymath}
            \qquad (1-\epsilon)\mathbb{P}(d(c_n, \emptyset) < \epsilon) \, \leq \, \int_X\mathbb{E}2^{-|O_n(x)|} \,\mathrm{d}\mu(x).
        \end{displaymath} In particular, if the strong Kesten property holds for $\A_{\mu}(G) \rtimes G$ or $\A_{\mu}(R) \rtimes[R]$, then  $\int_X\mathbb{E}2^{-|O_n(x)|} \,\mathrm{d}\mu(x)$ decays subexponentially.
    \item[$(2)$]\label{weak} For every $\epsilon \in (0,1)$, \begin{displaymath}
            \qquad \tfrac{1}{2}\mathbb{P}\!\left(d(c_n, \emptyset) < \tfrac{\epsilon}{2} n\right) -e^{-\epsilon n/2} \, \leq \, \int_X\mathbb{P}(|O_n(x)| \leq  4 \epsilon n) \,\mathrm{d}\mu(x).
        \end{displaymath} In particular, if the Kesten property holds for $\A_{\mu}(G) \rtimes G$ or $\A_{\mu}(R) \rtimes[R]$, then $\int_X\mathbb{P}(|O_n(x)| \leq  \epsilon n) \,\mathrm{d}\mu(x)$ decays subexponentially for each $\epsilon>0$.
\end{enumerate} \end{theorem}	

As the reader may notice, for amenable $R$, Theorem~\ref{inv.orb.T} directly relates the Kesten property with the behavior of inverted orbits of the random walks on equivalence classes of $R$. We will use this connection in Corollary~\ref{counterexample_K} later.
    
We will use the following lemma in the proof of Theorem~\ref{inv.orb.T}(2).
    
    \begin{lemma}\label{ProbLemma}
    	Let $X_n$ be a random variable taking values in $\{1,\ldots,n+1\}$, and let $Y_n$ be a random variable with distribution $\Bin(X_n,1/2)$.
    	Let $\epsilon>0$.  Then \begin{displaymath}
            \mathbb{P}(Y_n<\epsilon n) \, \leq \, \mathbb{P}(X_n< 4\epsilon n)+ e^{-\epsilon n/2} .
        \end{displaymath}
    \end{lemma}
    \begin{proof}
    	We may assume that $\epsilon<1/4$ and $\epsilon n >1$.  By definition of $Y_n$ we have 
    	\begin{align*}
    		&\mathbb{P}(Y_n<\epsilon n) \, = \, \sum\nolimits_{k=1}^{n+1}\mathbb{P}(\Bin(k,1/2)<\epsilon n) \mathbb{P}(X_n=k) \\
    		& \qquad = \, \sum\nolimits_{k=1}^{n+1}\mathbb{P}(X_n \leq k)(\mathbb{P}(\Bin(k,1/2) < \epsilon n)-\mathbb{P}(\Bin(k+1,1/2)<\epsilon n)) \\ & \hspace{60mm} +  \mathbb{P}(\Bin(n+2,1/2)<\epsilon n). 
    	\end{align*}
    	We can split the sum in the second line into two parts at $k=\lfloor 4\epsilon n\rfloor+1$, and we may bound the first part as
    	\begin{align*}
    		&\sum\nolimits_{k=1}^{\lfloor4\epsilon n\rfloor}\mathbb{P}(X_n \leq k)(\mathbb{P}(\Bin(k,1/2) < \epsilon n)-\mathbb{P}(\Bin(k+1,1/2)<\epsilon n)) \\
    		& \quad \leq \, \sum\nolimits_{k=1}^{\lfloor4\epsilon n \rfloor }\mathbb{P}(X_n \leq [4\epsilon n])(\mathbb{P}(\Bin(k,1/2) < \epsilon n)-\mathbb{P}(\Bin(k+1,1/2)<\epsilon n)) \\ 
    		& \quad = \, (1-\mathbb{P}(\Bin(\lfloor4\epsilon n\rfloor+1,1/2)<\epsilon n) \mathbb{P}(X_n \leq \lfloor4\epsilon n\rfloor)  .      	
    	\end{align*}
    	Notice that $\mathbb{P}(\Bin(k,1/2) < \epsilon n)-\mathbb{P}(\Bin(k+1,1/2)<\epsilon n) $ is always non-negative. So we can bound the second part as
    	\begin{align*}
    		&\sum\nolimits_{k=\lfloor4\epsilon n\rfloor+1}^{n+1}\mathbb{P}(X_n \leq k)(\mathbb{P}(\Bin(k,1/2) < \epsilon n)-\mathbb{P}(\Bin(k+1,1/2)<\epsilon n)) \\
            & \qquad \leq \, \sum\nolimits_{k=\lfloor4\epsilon n\rfloor+1}^{n+1} (\mathbb{P}(\Bin(k,1/2) < \epsilon n)-\mathbb{P}(\Bin(k+1,1/2)<\epsilon n)) \\
            & \qquad = \, \mathbb{P}(\Bin(\lfloor4\epsilon n\rfloor+1,1/2)<\epsilon n) -\mathbb{P}(\Bin(n+2,1/2)<\epsilon n) .
    	\end{align*}
    	Combining these bounds we obtain the following inequality
    	\begin{align*}
    		&\sum\nolimits_{k=1}^{n+1}\mathbb{P}(X_n \leq k)(\mathbb{P}(\Bin(k,1/2) < \epsilon n)-\mathbb{P}(\Bin(k+1,1/2)<\epsilon n)) \\
            & \hspace{60mm} + \mathbb{P}(\Bin(n+2,1/2)<\epsilon n) \\ 
            & \qquad \leq \, (1-\mathbb{P}(\Bin( \lfloor4\epsilon n \rfloor+1,1/2)<\epsilon n)) \mathbb{P}(X_n \leq \lfloor 4\epsilon n \rfloor) \\
            & \hspace{60mm} + \mathbb{P}(\Bin(\lfloor 4\epsilon n \rfloor +1,1/2)<\epsilon n)\\
            & \qquad \leq \, \mathbb{P}(X_n \leq \lfloor 4\epsilon n \rfloor) +\mathbb{P}(\Bin(\lfloor4\epsilon n \rfloor+1,1/2)<\epsilon n) .
    	\end{align*}
    	Now, by Hoeffding's inequality,
    	$$\mathbb{P}(\Bin(\lfloor 4\epsilon n \rfloor+1,1/2)<\epsilon n) \, \leq \,  e^{-2\frac{1}{16} 4\epsilon n} \, = \,  e^{-\epsilon n/2},$$
    	and the conclusion follows.
    \end{proof}
    
    Now we are ready to prove the theorem.

	\begin{proof}[Proof of Theorem~\ref{inv.orb.T}] 

We will use the following notation throughout the proof. For a fixed sequence of increments $\hat{h}=(\hat{h}_1,\ldots,\hat{h}_n)  \in  (\A_{\mu}(R) \rtimes [R])^n $ we will denote by $c_{\hat{h}}$  the  $\A_{\mu}(R)$-component of $\hat{g}_n=\hat{h}_n\cdots\hat{h}_1$.
\vspace{0.5em}

(1)	It is well-known (see, for example,  the proof of \cite[Proposition 5.27]{J}) that for almost every $x \in X,$ 
		$$\mathbb{P}(c_n(x)=\emptyset) \, = \, \mathbb{E} 2^{-|O_n(x)|}.$$
    
		Therefore, by Fubini's theorem, \begin{displaymath}
    \int_X \mathbb{E} 2^{-|O_n(x)|} \,\mathrm{d}\mu(x) \, = \, (\mu \otimes \hat{\nu}^{ \otimes n})(\{(x, \hat{h}) \in X \times (\A_{\mu}(G) \rtimes G)^n \mid c_{\hat{h}}(x)=\emptyset \}) .
    \end{displaymath}
		In particular,
		\begin{align*}
			&\int_X \mathbb{E} 2^{-|O_n(x)|} \,\mathrm{d}\mu(x) \\
            & \qquad \geq \, (\mu \otimes \hat{\nu}^{ \otimes n})(\{(x, \hat{h}) \in X \times (\A_{\mu}(G) \rtimes G)^n \mid c_{\hat{h}}(x)=\emptyset, d(\emptyset, c_{\hat{h}})<\epsilon \}) .
		\end{align*}
		However, since $1-d(\emptyset, c_{\hat{h}})\leq \mu (\{x \in X \mid c_{\hat{h}}(x)=\emptyset\})$,  Fubini's theorem implies that the right-hand side is greater than or equal to $(1-\epsilon)\mathbb{P}(d(c_n, \emptyset)<\epsilon)$.

		If the strong Kesten property holds for $\A_{\mu}(R)\rtimes[R]$ (or $\A_{\mu}(G) \rtimes G$), then we can apply it to $U=B_d(\epsilon) \times [R]$ (or $B_d(\epsilon) \times G$), and the random walk defined by $\hat{\nu}$. As a result, the above inequality implies that   $\int_X\mathbb{E}2^{-|O_n(x)|} \,\mathrm{d}\mu(x)$ decays subexponentially.

		(2)	Notice that for each $x$, the random variable $|c_n(x)|$ has binomial distribution $\Bin(|O_n(x)|,1/2) $. 
		
		For each $x\in X$, applying Lemma~\ref{ProbLemma} to $|O_n(x)|$ and $|c_n(x)|$, we see that \begin{displaymath}
		\int_X\mathbb{P}(|O_n(x)| \leq 4 \epsilon n) \,\mathrm{d}\mu(x) \, \geq \, \int_X\mathbb{P}(|c_n(x)| \leq  \epsilon n) \,\mathrm{d}\mu(x) -e^{-\epsilon n/2}. \end{displaymath} 
		By Fubini's theorem,
		\begin{align*}
		&\int_X\mathbb{P}(|c_n(x)| \leq  \epsilon n) \,\mathrm{d}\mu(x) \\
        & \quad = \, (\mu \otimes \hat{\nu}^{ \otimes n})(\{(x, \hat{h}) \in X \times (\A_{\mu}(G) \rtimes G)^n \mid |c_{\hat{h}}(x)| \leq  \epsilon n\}) \\
        & \quad \geq \, (\mu \otimes \hat{\nu}^{ \otimes n})(\{(x, \hat{h}) \in X \times (\A_{\mu}(G) \rtimes G)^n \mid |c_{\hat{h}}(x)| \leq \epsilon n, \, d(c_{\hat{h}}, \emptyset )\leq \tfrac{\epsilon}{2}n\} ) \\
        & \quad = \, \sum\nolimits_{\hat{h}:d(c_{\hat{h}}, \emptyset)\leq \frac{\epsilon}{2}n} \hat{\nu}^{ \otimes n}(\hat{h}) \mu(\{x \in X \mid |c_{\hat{h}}(x)| \leq  \epsilon n\}) \\
        & \quad \geq \, \tfrac{1}{2} \sum\nolimits_{\hat{h}:d(c_{\hat{h}}, \emptyset)\leq  \frac{\epsilon}{2}n} \hat{\nu}^{ \otimes n}(\hat{h})  \\
        & \quad = \, \tfrac{1}{2}\mathbb{P}\!\left(d(c_n, \emptyset) < \tfrac{\epsilon}{2} n\right) .
		\end{align*} 
     The last inequality follows from Markov's inequality applied to the map $x \mapsto |c_{\hat{h}}(x)|$ and the assumption \[ \int_X |c_{\hat{h}}(x)| d \mu(x)= d(c_{\hat{h}}, \emptyset)\leq  \frac{\epsilon}{2}n. \] 
		Thus, if the Kesten property is satisfied in this setting, we can set
        \begin{displaymath}
    U \, = \, B_d\!\left(\tfrac{\epsilon}{2}\right) \times V ,
\end{displaymath} where $V$ is some neighborhood of the identity in $G$, and \begin{displaymath}
    B_d\!\left(\tfrac{\epsilon}{2}\right) \, = \, \{c \in C(G) \mid d(c, \emptyset) < \tfrac{\epsilon}{2} \} ,
\end{displaymath} and deduce that the probability $\mathbb{P}(d(c_n, \emptyset) < \frac{\epsilon}{2} n) $ decays subexponentially, so \begin{displaymath}
    \tfrac{1}{2}\mathbb{P}\!\left(d(c_n, \emptyset) < \tfrac{\epsilon}{2} n\right) -e^{-\epsilon n/2}
\end{displaymath} decays subexponentially as well. \end{proof}

\begin{corollary}\label{counterexample_K}
Let  $(X, \mu)$  be the Poisson boundary of the simple random walk\footnote{The simple random walk on $F_2=F_2(a,b)$ is the one assigning the weight $1/4$ to $a,b, a^{-1}, b^{-1}$.} on the free group $F_2$ and let $R$ be the equivalence relation generated by the natural action of $F_2$ on $(X, \mu)$. Then the measurable lamplighter $\A_{\mu}(R) \rtimes [R]$ is a contractible, hence locally generated, amenable Polish topological group without the Kesten property.

\end{corollary}

 \begin{proof}
Consider the natural embedding of $F_2$ into $[R]$, and let $\nu$ be the measure corresponding to the 
simple random walk on $F_2$.
Since the action of $F_2$ on $(X, \mu)$ is free almost surely, the quantity $\mathbb{P}(|O_n(x)| \leq  \epsilon n)$ is almost surely equal to the probability that the inverted orbit of the identity in $F_2$ under the first $n$ steps of the simple random on $F_2$ driven by $\nu$ is at most $\epsilon n$. Since $F_2$ is not amenable, the left multiplication action of $F_2$ on itself is not extensively amenable, so the latter quantity has exponential decay as $n \rightarrow \infty$ by Proposition~\ref{inv_orb_char}(2). Therefore, the integrals
$\int_{X}\mathbb{P}(|O_n(x)| \leq  \epsilon n) \,\mathrm{d}\mu(x)$ decay exponentially as $ n \rightarrow \infty$, and as a result, Theorem~\ref{inv.orb.T} implies that the lamplighter $\A_{\mu}(R) \rtimes[R]$ cannot satisfy the Kesten property. Amenability of $R$ with respect to $\mu$ is proved in \cite{Zimmer}, so $\A_{\mu}(R) \rtimes[R]$ is amenable by Theorem \ref{Top.Lamplighter}(2). Finally, as $\mu$ is $R$-ergodic, $\A_{\mu}(R) \rtimes[R]$ is contractible by Theorem~\ref{Top.Lamplighter}(3), and in particular, it is path-connected, thus locally generated. \end{proof}

The reader might notice that the argument in Corollary \ref{counterexample_K} heavily relies on non-amenability of the actions of $F_2$ on the orbits of points in $X$, but this is possible only because we do not require the measure $\mu$ to be invariant. As we already mentioned in Section~\ref{section:extensive.amenability}, when $R$ is $\mu$-amenable and the measure $\mu$ is assumed to be $R$-invariant, the actions of $G \leq [R]$ on the orbits of $R$ tend to behave similarly to the actions of $G$ on its quotients modulo normal co-amenable subgroups, and one may expect the measurable lamplighter to satisfy the Kesten property in the measure-preserving case.

Let us now describe potential ways to connect the subexponential decay of the integrals $\int_X\mathbb{E}2^{-|O_n(x)|}\,\mathrm{d}\mu(x)$ or $\int_X\mathbb{P}(|O_n(x)| \leq \epsilon n) \,\mathrm{d}\mu(x)$ with the extensive amenability of actions on orbits.  It is well-known that for a point $y$ in a $G$-set $Y$, $\mathbb{E}2^{-|O_n(y)|} $ is equal to the return probability of a random walk on $\mathcal{P}_f(Y)$ induced by a switch-walk-switch random walk on $\mathcal{P}_f(Y)\rtimes G$ with the ``switching'' component corresponding to $1/2(\delta_{\emptyset}+\delta_{y})$. Hence, the limit $$\rho_y \, = \, \lim\nolimits_{n \rightarrow \infty}\left(\mathbb{E}2^{-|O_n(y)|}\right)^{1/n}$$ exists for every $y \in Y$, and moreover, it is equal to the spectral radius of the switch-walk-switch random walk (corresponding to $y$) on $\mathcal{P}_f(Y)$.

Moreover, as we explained in Remark \ref{remark:limit_for_inv_orbits}, for any action of a group $G$ on a set $Y$, any $ \epsilon >0$, any $y \in Y$, and any symmetric probability measure $\mu$ on $G$, the limit $$p_{\epsilon}(y) \, = \, \lim\nolimits_{n \rightarrow \infty } {\mathbb{P}  \left(|O_n(y)| \leq  \epsilon n\right)}^{1/n}$$ exists as well. The following proposition relates independence of $p_{\epsilon}(y)$ and $\rho_y$ from the choice of $y$ in the $G$-orbit with the rate of decay of the corresponding integrals. 

\begin{proposition}
 Let $G$ be a countable group acting on a standard measure space $(X, \mu)$ in a Borel way such that $\mu$ is quasi-invariant and ergodic.  Consider any symmetric random walk on $G$ and let us retain the notation for inverted orbits and limits defined above.
 Then the following hold.
\begin{enumerate}
    \item Assume that for some $\epsilon >0 $, for almost every $x \in X$ the value of $p_{\epsilon}(x)$  is constant on the $G$-orbit of $x$. Then if the integral $ \int_X\mathbb{P}(|O_n(x)| \leq \epsilon n) \,\mathrm{d}\mu(x) $ decays subexponentially, we have $p_{\epsilon}(x)=1$ almost surely on $X$.
    \item Assume that $\rho_y$ is constant on the $G$-orbit of $y$ for a.e. $y \in X$, and the integral $ \int_X\mathbb{E}2^{-|O_n(x)|}\,\mathrm{d}\mu(x)$ decays subexponentially. Then $\rho_y =1$ almost surely on~$X$.
   
\end{enumerate}
  
\end{proposition}

\begin{proof}
We prove~(1), the proof of~(2) is similar. Since $p_{\epsilon}(x)$ is a measurable function on $X$ which is constant on the orbits of $G$, ergodicity of the action of $G$ on $(X, \mu)$ implies that $p_{\epsilon}(x)$ is constant almost everywhere on $X$. Let us denote its value by $p_{\epsilon}$.
Recall that in Remark \ref{remark:limit_for_inv_orbits} we showed that the function $n \mapsto {\mathbb{P} \left(|O_n(x)| \leq  \epsilon n\right)}$ is supermultiplicative with respect to $n$  and that ${\mathbb{P} \left(|O_n(x)| \leq  \epsilon n\right)} \leq p_{\epsilon}(x)^n =p_{\epsilon}^n$ for almost every $x$. Hence, if $p_{\epsilon} < 1$, then we have $ \int_X\mathbb{P}(|O_n(x)| \leq \epsilon n) \,\mathrm{d}\mu(x)  \leq p_{\epsilon}^n$ and the integral decays exponentially, which contradicts the assumption.  
\end{proof}

As a corollary, a positive answer to the following question means that extensive amenability of action on almost every orbit may be deduced from Theorem~\ref{inv.orb.T}, provided that the relevant version of the Kesten property holds in this setting. 

	\begin{question} \label{Q: inv_orbit}
		Let $G$ be a countable group acting transitively on a set $Y$ and consider any symmetric random walk on $G$. 
        \begin{enumerate}
            \item Is any of the quantities $\rho_y$ and  $p_{\epsilon}(y)$ defined above independent of the choice of $y \in Y$?
            \item Is any of these limits independent of a choice of $y$ under the assumption that $G$ is finitely generated,  the random walk is driven by a measure with finite support, and the orbital Schreier graph is amenable or, even stronger, has polynomial uniform growth?
        \end{enumerate}  
	\end{question}
We also note that the connection between $\rho_y$ and the spectral radius of the switch-walk-switch random walk may make this question easier to answer for $\rho_y$.

\begin{remark}
Assume that a finitely generated amenable group $G$ acts freely on a standard measure space $(X,\mu)$ in a Borel way so that $\mu$ is invariant and ergodic, and let $R$ be the associated orbit equivalence relation. Then the lamplighter construction still produces a Polish group if the full group $[R]$ is replaced with the $L^1$ full group of the action of $G$, see \cite{LeMaitre18} and \cite{LeMaitre21} for a detailed study of the $L^1$ full groups. The $L^1$ full groups are equipped with natural metric turning them into extremely amenable Polish groups, and they enjoy a variety of rigidity properties. In particular, any abstract isomorphism of the $L^1$ full groups of two ergodic actions turns out to be a quasi-isometry of respective metrics, see~\cite[Theorem~C]{LeMaitre21}. The group $\IET(\Lambda)$ embeds into the $L^1$ full group of the corresponding action of a free abelian group. Moreover, an analog of Theorem \ref{inv.orb.T} is also valid in this setting. Therefore, establishing Kesten property for lamplighters over the $L^1$ full groups can also become a step towards solving amenability problem for the $\IET$ group.
\end{remark}

\section{Appendix}\label{lcsc Kesten}

In this appendix we provide the proof of ~\cite[Corollary~7.3]{Q}, restricted to the measures $\mu$ with $1 \in \supp{\mu}$, for the reader's convenience. The additional assumption is technical and allows us to use \cite[Theorem 1]{BeCh-74a} directly. Furthermore, we also retain the notation from \cite{Q}, and in particular, for subset $S \subseteq G$, we denote by $1_S( \cdot)$ the indicator function of $S$. We do not claim any original results in this section.

The corollary may be stated as follows.

\begin{proposition}
Let $G$ be a compactly generated locally compact group and $\mu$ be a symmetric regular Borel probability measure on $G$ such that $1 \in \supp{\mu}$. Suppose that the support of $\mu$ generates a non-amenable subgroup of $G$. Let $K$ be a symmetric compact generating set of $G$. Then there exist $\alpha, \epsilon >0$ such that 
$\mu^n(K^{\lfloor\epsilon n\rfloor})$ decays faster than $e^{-\alpha n}$.
\end{proposition}
The proof uses Kesten's theorem
for locally compact groups, see, for example, \cite[Theorem 1]{BeCh-74a}. More precisely, it uses the fact that if the closed subgroup generated by the support of a Borel probability measure $\mu$ on $G$ is non-amenable, then the spectral radius of the Markov operator $P_{\mu}$ on $L^2(G)$ is strictly less than 1. Another important observation is the well-known fact that for the left-invariant Haar measure $\lambda$ on $G$  the sequence $\lambda(K^n)^{\frac{1}{n}}, \; n\geq 1, $ converges.

\begin{proof}
We may assume that $K$ has non-empty interior. Let $L$ be any compact subset of $G$. Then, for any $g,x \in G$ we have 
$$ 1_{L} (g) 1_{K}(x) \leq 1_{LK}(gx)1_K(x) .$$
Thus, for any $n$, integrating this inequality with respect to $\mu^n \times \lambda $ and using Fubini's theorem we obtain
\begin{displaymath}
	\mu^n(L)\lambda(K) \leq \int_{G} \bigg( \int_G 1_{LK}(gx)\,\mathrm{d}\mu^n(g)\bigg) 1_K(x) d\lambda(x) .
\end{displaymath} 
However, \begin{displaymath}
    \int_{G} \bigg( \int_G 1_{LK}(gx)\,\mathrm{d}\mu^n(g)\bigg) 1_K(x) \,\mathrm{d}\lambda(x)= \langle P_{\mu}^n 1_{LK}, 1_K\rangle_{L^2(G)}.
\end{displaymath} Since the closed subgroup generated by $\supp{\mu}$ is non-amenable, \cite[Theorem 1]{BeCh-74a} implies that $ \|P_{\mu} \| < 1$. Therefore, there exists $\alpha>0$ such that $ \|P_{\mu}^n\| < e^{-\alpha n}$  for all $n$, and so \begin{displaymath}
    \langle P_{\mu}^n 1_{LK}, 1_K\rangle_{L^2(G)} < e^{-\alpha n} \|1_{LK} \|_2 \|1_K \|_2=e^{-\alpha n}\lambda(LK) \lambda(K).
\end{displaymath} Since for every compact $C$ the sequence $\lambda(C^n)^{\frac{1}{n}}, \; n\geq 1, $ converges,  one may choose a sufficiently small $\epsilon$ such that for $L=K^{\lfloor \epsilon n \rfloor} $ the sequence  $\lambda(K^{\lfloor \epsilon n\rfloor+1}) $ grows slower than $e^{\alpha n/2 }$. Then we have  $$	\mu^n(K^{\lfloor\epsilon n\rfloor}) \leq  \langle P_{\mu}^n 1_{K^{\lfloor \epsilon n\rfloor+1}}, 1_K\rangle_{L^2(G)} / \lambda(K)< e^{-\alpha n /2} .$$ 
Thus, the desired inequality follows.
\end{proof}

\end{document}